\documentclass[11pt,a4paper]{article}

\usepackage{latexsym}
\usepackage{amsmath}
\usepackage{amssymb}
\usepackage{amsthm}
\usepackage{amscd}
\usepackage{mathrsfs}
\usepackage[all]{xy}
\usepackage{graphicx}
\usepackage{bm}
\usepackage{comment}
\usepackage{arydshln}
\usepackage{mathtools}
\usepackage{relsize}

\usepackage[usenames,dvipsnames]{color}

\newtheorem{thm}{Theorem}[section]
\newtheorem{prop}[thm]{Proposition}
\newtheorem{defn}[thm]{Definition}
\newtheorem{lem}[thm]{Lemma}
\newtheorem{cor}[thm]{Corollary}
\newtheorem{conj}[thm]{Conjecture}
\newtheorem{rem}[thm]{Remark}

\newtheorem*{thmn}{Theorem}

\makeatletter

\@addtoreset{equation}{section}
\makeatother

\makeatletter
\newcommand{\subsubsubsection}{\@startsection{paragraph}{4}{\z@}%
 {1.0\Cvs \@plus.5\Cdp \@minus.2\Cdp}%
 {.1\Cvs \@plus.3\Cdp}%
 {\reset@font\sffamily\normalsize}
 }
\makeatother
\newcommand{\cf}{\textit{cf.\ }}

\DeclareMathOperator{\Gal}{Gal}
\DeclareMathOperator{\id}{id}

\DeclareMathOperator{\Hom}{Hom}

\DeclareMathOperator{\Rep}{Rep}
\DeclareMathOperator{\Mod}{Mod}

\DeclareMathOperator{\Aut}{Aut} 
\DeclareMathOperator{\Fil}{Fil} 
\DeclareMathOperator{\Ind}{Ind}

\DeclareMathOperator{\Spa}{Spa}
\DeclareMathOperator{\Lie}{Lie}

\DeclareMathOperator{\GL}{GL}
\DeclareMathOperator{\basic}{basic}
\DeclareMathOperator{\supp}{supp}
\DeclareMathOperator{\IC}{IC}
\DeclareMathOperator{\Bun}{Bun}
\DeclareMathOperator{\Perf}{Perf}
\DeclareMathOperator{\Hecke}{Hecke}
\DeclareMathOperator{\Gr}{Gr}
\DeclareMathOperator{\diag}{diag}
\DeclareMathOperator{\gr}{gr}

\DeclareMathOperator{\sch}{sch}
\DeclareMathOperator{\Div}{Div}
\DeclareMathOperator{\et}{et}

\DeclareMathOperator{\Frob}{Frob}
\DeclareMathOperator{\rk}{rk}
\DeclareMathOperator{\Isom}{Isom}

\DeclareMathOperator{\lis}{lis}

\newcommand{\bA}{\mathbb{A}}
\newcommand{\bB}{\mathbb{B}}
\newcommand{\bC}{\mathbb{C}}
\newcommand{\bD}{\mathbb{D}}

\newcommand{\bF}{\mathbb{F}}
\newcommand{\bG}{\mathbb{G}}

\newcommand{\bQ}{\mathbb{Q}}

\newcommand{\bZ}{\mathbb{Z}}

\newcommand{\cC}{\mathcal{C}}
\newcommand{\cD}{\mathcal{D}}

\newcommand{\cF}{\mathcal{F}}

\newcommand{\cM}{\mathcal{M}}

\newcommand{\cO}{\mathcal{O}}
\newcommand{\cP}{\mathcal{P}}

\newcommand{\cS}{\mathcal{S}}
\newcommand{\cT}{\mathcal{T}}

\newcommand{\fm}{\mathfrak{m}}

\newcommand{\sE}{\mathscr{E}}
\newcommand{\sF}{\mathscr{F}}

\newcommand{\sT}{\mathscr{T}}
\newcommand{\sU}{\mathscr{U}}

\newcommand{\ol}{\overline}
\newcommand{\ul}{\underline}
\newcommand{\wh}{\widehat}
\newcommand{\wt}{\widetilde}
\newcommand{\lra}{\longrightarrow}
\newcommand{\xra}{\xrightarrow}
\newcommand{\ola}{\overleftarrow}
\newcommand{\ora}{\overrightarrow}

\newcommand{\solid}{\mathsmaller{\blacksquare}}

\begin{document}

\title{Non-semi-stable loci in Hecke stacks and Fargues' conjecture}
\author{Ildar Gaisin and Naoki Imai}
\date{}
\maketitle

\begin{abstract}
We show the Harris--Viehmann conjecture 
under some Hodge--Newton reducibility condition 
for a generalization of the diamond of 
a non-basic Rapoport--Zink space at infinite level, 
which appears as a cover of the non-semi-stable locus 
in the Hecke stack. 
We show also that 
the cohomology of the non-semi-stable locus 
with coefficient coming from a cuspidal Langlands parameter 
vanishes. 
As an application, we show the Hecke eigensheaf property in 
Fargues' conjecture for cuspidal Langlands parameters 
in the $\GL_2$-case. 
\end{abstract}

\section*{Introduction}
In \cite{FarGover}, Fargues formulated a conjecture on 
a geometrization of the local Langlands correspondence 
motivated by a formulation of the geometric Langlands conjecture 
in \cite{FGVGeo}.  

Let $E$ be a $p$-adic number field with 
residue field $\bF_q$. 
Let $G$ be a quasi-split reductive group over $E$. 
Then we can define a moduli stack $\Bun_G$ 
of $G$-bundle on the Fargues--Fontaine curve, 
and a moduli $\Div_X^1$ of Cartier divisors of degree $1$ on 
the Fargues--Fontaine curve. 
Further, we have a diagram 
\[
 \xymatrix{
 & \Hecke^{\leq \mu} 
 \ar@{->}[ld]_-{\ola{h}} \ar@{->}[rd]^-{\ora{h}} & \\ 
 \Bun_G  & & 
 \Bun_G \times \Div_X^1 , 
 } 
\]
where 
$\Hecke^{\leq \mu}$ 
is a moduli stack of modifications of $G$-bundle on 
the Fargues--Fontaine curve 
with some condition determined by 
a cocharacter $\mu$ of $G$, 
which is called a Hecke stack. 
For a discrete Langlands parameter 
$\varphi \colon W_E \to {}^L G$, 
Fargues' conjecture predicts 
the existence of a sheaf $\sF_{\varphi}$ 
on $\Bun_{G}$ 
satisfying some conditions, 
the most intriguing one of which is 
the Hecke eigensheaf property 
\[
 \ora{h}_{\natural} (\ola{h}^* 
 \sF_{\varphi} \otimes \IC_{\mu}') = 
 \sF_{\varphi} \boxtimes (r_{\mu} \circ \varphi ) , 
\]
where 
$r_{\mu}$ 
is a representation of ${}^L G$ 
determined by $\mu$, and 
$\IC_{\mu}'$ is an object of the derived category of sheaves determined by $\mu$ via the geometric Satake 
correspondence. 
The conjecture is stated based on some conjectural objects. 
However, in the case $\varphi$ is cuspidal and $\mu$ is minuscule, 
we can define every object in the conjecture 
assuming only the local Langlands correspondence, 
which is constructed in many cases.

Assume that 
$\varphi$ is cuspidal and $\mu$ is minuscule. 
Then the support of the sheaf 
$\sF_{\varphi}$ is contained in 
the semi-stable locus $\Bun_{G}^{\mathrm{ss}}$ of 
$\Bun_{G}$. 
The Hecke eigensheaf property then predicts 
that 
\[
 \supp \ora{h}_{\natural} (\ola{h}^* 
 \sF_{\varphi} \otimes \IC_{\mu}') \subset 
 \Bun_{G}^{\mathrm{ss}} \times \Div_X^1 . 
\]
This is non-trivial since the inclusion 
\[
 \ola{h}^{-1} \bigl( \Bun_{G}^{\mathrm{ss}} \bigr)  
 \subset 
 \ora{h}^{-1} \bigl( \Bun_{G}^{\mathrm{ss}} 
 \times \Div_X^1 \bigr) 
\]
does not hold. 
The vanishing of 
$\ora{h}_{\natural} (\ola{h}^* \sF_{\varphi} \otimes \IC_{\mu}')$ 
outside the semi-stable locus 
involves geometry of a non-semi-stable locus of 
the Hecke stack $\Hecke^{\leq \mu}$. 

One aim of this paper is to give 
a partial result in this direction. 
Assume that $\varphi$ is cuspidal, 
but $\mu$ can be general in the following. 
Let $B(G)$ be the set of 
$\sigma$-conjugacy classes 
in $G(\breve{E})$, where 
$\breve{E}$ is the completion of 
the maximal unramified extension of $E$. 
Then we have a decomposition 
\[
 \Bun_{G} = 
 \coprod_{[b] \in B(G)} 
 \Bun_{G}^{[b]} 
\]
into strata, 
where the the strata corresponding to 
basic elements of $B(G)$ forms the 
semi-stable locus. 
Let $[b],[b'] \in B(G)$. 
We define 
$\Hecke_{[b],[b']}^{\leq \mu}$ 
by the fiber products 
\begin{equation*}
 \xymatrix{
 \Hecke_{[b],[b']}^{\leq \mu} \ar@{->}[r] \ar@{->}[dd] 
 & 
 \Hecke_{[b]}^{\leq \mu} \ar@{->}[r] \ar@{->}[d] & 
 \Bun_{G}^{[b]} \times \Div_X^1 
 \ar@{->}[d] \\
  & 
 \Hecke^{\leq \mu} \ar@{->}^-{\overrightarrow{h}}[r] 
 \ar@{->}^-{\overleftarrow{h}}[d] & 
 \Bun_{G} \times \Div_X^1 \\ 
 \Bun_{G}^{[b']} \ar@{->}[r] & 
 \Bun_{G}. &  
 }
\end{equation*}
We assume that $[b]$ is not basic, 
and $[b']$ is basic. 
Let $\Hecke_{[b],[b']}^{\mu}$ 
be an open substack of 
$\Hecke_{[b],[b']}^{\leq \mu}$, 
where the modifications have type $\mu$. 
We find that 
a generalization $\cM_{b,b'}^{\mu}$ of 
a diamond of a non-basic Rapoport--Zink space 
at infinite level 
covers 
$\Hecke_{[b],[b']}^{\mu}$. 

We can define a Levi subgroup $L^b$ of $G$ 
such that 
$[b]$ is an image of a basic element 
$[b_{00}]$ of $B(L^b)$. 
Take a proper Levi subgroup $L$ of $G$ 
containing $L^b$. 
Let $[b_0]$ be the image of $[b_{00}]$ 
in $B(L)$. 
We assume that 
$[b']$ is in the image of 
an element $[b'_0] \in B(L)$. 
Further, we assume that $([b],[b'],\mu)$ 
satisfies a twisted analogue of 
Hodge--Newton reducibility. 
Our main theorem is the following: 
\begin{thmn}
The compactly supported cohomology of 
$\cM_{b,b'}^{\mu}$ 
is a parabolic induction of 
the compactly supported cohomology of 
$\cM_{b_0,b'_0}^{\mu}$ 
with some degree shift and twist. 
\end{thmn}
See Theorem \ref{thm:HInd} for the precise statement. 
This theorem is a generalization of 
the Harris--Viehmann conjecture on 
cohomology of non-basic Rapoport--Zink spaces 
in \cite[Conjecture 8.5]{RVlocSh} 
(\cf \cite[Conjecture 5.2]{HarLLCvcS}) up to a character twist
under the Hodge--Newton reducibility condition. 
We also show that the compactly supported cohomology of 
$\cM_{b,b'}^{\mu}$ does not contain any supercuspidal representation. 
These results can be viewed as generalization of results in \cite{MantnbRZ}. 
Using the above theorem, we can show the following: 
\begin{thmn}
The compactly supported cohomology of 
$\Hecke_{[b],[b']}^{\mu}$ with coefficient in 
$\ola{h}^* \sF_{\varphi}$ 
vanishes. 
\end{thmn}
See Theorem \ref{thm:van} 
for the precise statement. 
This result is partial, since we are assuming Hodge--Newton reducibility. 
On the other hand, 
the assumption is automatically satisfied 
if 
$\Hecke_{[b],[b']}^{\leq \mu}$ is not empty 
in the case where $G=\GL_2$ and 
$\mu (z) =\diag (z,1)$. 
As an application, we can show the following: 
\begin{thmn}
Assume that $G=\GL_2$ and 
$\mu (z) =\diag (z,1)$. 
Then the Hecke eigensheaf property for a cuspidal Langlands 
parameter holds. 
\end{thmn}

During the course of this work, 
Hansen put a related preprint 
\cite{HanHarr} on his webpage, 
which shows the Harris--Viehmann conjecture for $\GL_n$ under 
the Hodge--Newton reducibility condition. 
We learned 
his result on canonical filtrations 
and some consequences of 
Scholze's work \cite{SchEtdia} on cohomology of diamonds  
from \cite{HanHarr}. 
Note that the result of \cite{HanHarr} 
is enough for the application to Fargues' conjecture in $\GL_2$-case. 
Our main points are 
proving the Harris--Viehmann conjecture 
under the Hodge--Newton reducibility condition 
for general reductive groups 
and making the relation to Fargues' conjecture clear. 
After this work was done, 
Fargues' conjecture 
for cuspidal Langlands parameters in the $\GL_n$-case is proved in \cite{AnLBAvGLn} and \cite{HansclocSh} by a different method. 

In Section \ref{sec:Gbun}, 
we recall a definition of 
the stack of $G$-bundle on the Fargues--Fontaine curve, 
and its structure. 
In Section \ref{sec:Hsta}, 
we recall a definition of the Hecke stack and explain a cohomological formula. 
In Section \ref{sec:Fconj}, 
we construct a sheaf which satisfies 
properties (1), (2) and (3) of \cite[Conjecture 4.4]{FarGover} 
and explain the Hecke eigensheaf property 
in Fargues' conjecture 
for cuspidal Langlands parameters. 

In Section \ref{sec:vcoh}, 
we study a non-semi-stable locus in the Hecke stack. 
We find that a generalization of 
a diamond of a non-basic Rapoport--Zink space at infinite level 
covers the non-semi-stable locus in the Hecke stack. 
We show that the cohomology of 
the generalized space 
can be written as a parabolic induction of 
the cohomology of smaller space 
associated a Levi subgroup 
under the Hodge--Newton reducibility condition. 
In particular, 
we see that the cohomology 
does not contain any supercuspidal representation 
in each degree. 
As a result, we show that 
the cohomology of the non-semi-stable locus in the Hecke stack 
with a coefficient coming from a cuspidal Langlands parameter 
vanishes. 

In Section \ref{sec:NALT}, 
we see that we can recover  
Hecke eigensheaf property 
on some part of the semi-stable locus 
from non-abelian Lubin--Tate theory 
in the $\GL_n$-case. 
In Section \ref{sec:Hep}, 
we show that the Hecke eigensheaf property 
in the $\GL_2$-case, 
using the results in the preceding sections. 

\subsection*{Acknowledgements}
The authors would like to thank Laurent Fargues 
and Peter Scholze 
for answering our questions on 
their forthcoming works. 
They also want to thank Paul Ziegler 
for answering a question regarding his work. 
They are grateful to David Hansen 
for his helpful comment 
on a previous version of this paper. 
They also want to thank Teruhisa Koshikawa 
for his comments on this paper. 
Finally, they thank a referee for helpful comments and suggestions. 

\section{Stack of $G$-bundles}\label{sec:Gbun}

In this section we recall various results regarding the stack of $G$-bundles on the curve. Let $p$ be a prime number. 
Fix $E$ a finite extension of $\bQ_p$ 
with residue field $\bF_q$. 
We follow the definition of perfectoid algebra in 
\cite[1.1]{FonPerBou} (\cf \cite[Definition 5.1]{SchPerf}). 
For an algebraic extension $k$ of $\bF_q$, let $\Perf_{k}$ be the category of 
perfectoid spaces over $k$ 
equipped with v-topology (\cf \cite[Definition 8.1(iii)]{SchEtdia}). 
For $S \in \Perf_{\bF_q}$, 
we have the relative Fargues--Fontaine curve 
$X_S = Y_{S}/\varphi^{\bZ}$ as in \cite[Definition II.1.15]{FaScGeomLLC}. 
For an affinoid perfectoid $\Spa (R,R^+) \in \Perf_{\bF_q}$, 
we have also the schematical 
relative Fargues--Fontaine curve 
$X_{\Spa(R,R^{+})}^{\sch}$ as defined just after \cite[Remark II.2.8]{FaScGeomLLC}. The schematic version $X_{\Spa(R,R^{+})}^{\sch}$ only depends on $R$ and so we denote it by $X_R^{\sch}$. 
We have an equivalence between categories of 
vector bundles on 
$X_{\Spa (R,R^+)}$ and $X_R^{\sch}$ 
by \cite[Theorem 8.7.7]{KeLiRpHF}.

Let $G$ a connected reductive group over $E$. 
Let $\Bun_G$ be the fibered category in groupoids 
whose fiber at $S \in \Perf_{\ol{\bF}_q}$ 
is the groupoid of $G$-bundles on $X_S$. 
Then $\Bun_G$ has a reasonable geometry. Let us just mention that, in particular it is a small v-stack  (\cf \cite[Proposition III.1.3]{FaScGeomLLC}). 

Let $\breve{E}$ be the completion 
of the maximal unramified extension of $E$. 
Let $\sigma$ be the continuous automorphism of 
$\breve{E}$ lifting the $q$-th power Frobenius 
on the residue field. 
For $b \in G(\breve{E})$, 
we have an associated 
$G$-isocrystal 
\[
 \cF_b \colon 
 \Rep (G) \lra \varphi\mathchar`-\Mod_{\breve{E}} ; \ 
 (V,\rho) \mapsto (V \otimes_E \breve{E}, \rho(b)\sigma ). 
\]
Let $B(G)$ be the set of $\sigma$-conjugacy classes in 
$G(\breve{E})$. 
Then 
we have a bijection 
\[
 B(G) \lra \{ 
 \textrm{the isomorphism classes of 
 $G$-isocrystals over $\breve{E}$} \} ;\ 
 [b] \mapsto [\cF_b] 
\]
by \cite[Remarks 3.4 (i)]{RaRiFiso}. 

Let $S \in \Perf_{\ol{\bF}_q}$. 
We have a functor 
\[
 \varphi\mathchar`-\Mod_{\breve{E}} \lra 
 \Bun_{X_S} ;\ 
 (D,\varphi) \mapsto \sE (D,\varphi ), 
\]
where $\sE (D,\varphi )$ is given by 
\[
 Y_S \times_{\varphi} D \lra 
 Y_S /\varphi^{\bZ} =X_S. 
\]
The composite 
\[
 \Rep (G) \stackrel{\cF_b}{\lra} 
 \varphi\mathchar`-\Mod_{\breve{E}} 
 \xra{\sE(-)} \Bun_{X_S} 
\]
gives a $G$-bundle $\sE_{b,X_S}$ on $X_S$. 
We simply write $\sE_b$ 
for $\sE_{b,X_S}$ sometimes. 
If $b'=g b \sigma (g)^{-1}$, 
then we have an isomorphism 
\begin{equation}\label{eq:tg}
 t_g \colon \sE_{b,X_S} \lra 
 \sE_{b',X_S}  
\end{equation}
induced by the multiplication by $g$. 
The isomorphism class of $\sE_{b,X_S}$ 
depends only on the class of $b$ in $B(G)$. Moreover by \cite[Theorem III.2.2]{FaScGeomLLC}, this gives a complete description of the points of $\Bun_G$. 

Let $\pi_1 (G)$ be an algebraic fundamental group 
of $G$ defined in \cite[1.4]{BoroAGc}. 
Let $\ol{E}$ be a separable closure of $E$ and let $\Gamma =\Gal (\ol{E}/E)$ be its absolute Galois group. 
Let 
\[
 \kappa \colon 
 B(G) \lra \pi_1 (G)_{\Gamma}  
\]
be the Kottwitz map 
in \cite[Theorem 1.15]{RaRiFiso} 
(\cf \cite[Lemma 6.1]{KotShlam}). Then \cite[Theorem III.2.7]{FaScGeomLLC}) provides a decomposition 
\[
 \Bun_{G} = 
 \coprod_{\alpha \in \pi_1 (G)_{\Gamma}} 
 \Bun_{G}^{\alpha} 
\]
into open and closed substacks. 

Let $\bD$ be the split pro-algebraic torus 
over $E$ such that 
$X_* (\bD)=\bQ$. 
For $b \in G(\breve{E})$, 
we have an associated homomorphism 
\[
 \tilde{\nu}_b \colon \bD_{\breve{E}} \lra G_{\breve{E}} 
\]
constructed in \cite[4.2]{KotIso}. 
This gives a well-defined map 
\[
 \nu \colon 
 B(G) \lra \bigl( 
 \Hom (\bD_{\breve{E}},G_{\breve{E}})/G(\breve{E}) 
 \bigr)^{\sigma} ;\ 
 [b] \mapsto [\tilde{\nu}_b], 
\]
which is called the Newton map. 
We say that 
$b \in G(\breve{E})$ is basic, if 
$\tilde{\nu}_b$ factors through the center of 
$G_{\breve{E}}$. 
We say that $[b] \in B(G)$ is basic 
if it consists of basic elements in $G(\breve{E})$. 
Let $B(G)_{\basic}$ denote 
the basic elements in $B(G)$. 
We recall that 
the Kottwitz map induces a bijection 
\[
 \kappa \colon B(G)_{\basic} \stackrel{\sim}{\lra} 
 \pi_1 (G)_{\Gamma} . 
\]

Assume that $G$ is quasi-split in the sequel.  
We fix subgroups $A \subset T \subset B$ 
of $G$, where $A$ is a maximal split torus,
$T$ is a maximal torus and 
$B$ is a Borel subgroup. 
We write 
$X_* (A)^+$ for the dominant cocharacters of $A$. 
Then we have a natural isomorphism 
\[
 X_* (A)^+_{\bQ} \stackrel{\sim}{\lra} 
 \bigl( 
 \Hom (\bD_{\breve{E}},G_{\breve{E}})/G(\breve{E}) 
 \bigr)^{\sigma}. 
\]
Let $b \in G(\breve{E})$. 
We write 
$\nu_b \in X_* (A)^+_{\bQ}$ 
for the representative of $[\tilde{\nu}_b]$. 
Let $w$ be the maximal length element 
in the Weyl group of $G$ with respect to $T$. 
Then the map 
\[
 \mathrm{HN} \colon 
 B(G) \to X_* (A)^+_{\bQ} ; \ 
 [b] \mapsto w \cdot (- \nu_b) 
\]
is called the Harder--Narasimhan map. After equipping $X_* (A)^+_{\bQ}$ with the natural order topology, as discussed in \cite[Section 2]{RaRiFiso}, the map $\mathrm{HN}$ is upper semicontinuous by \cite[Theorem III.2.3]{FaScGeomLLC}. 

We define an algebraic group $J_b$ over $E$ by 
\[
 J_b (R)= \{ g \in G(R \otimes_E \breve{E} ) \mid 
 g b \sigma (g)^{-1} = b \} 
\]
for any $E$-algebra $R$. 
Then we have $J_b (E)=\Aut (\cF_b)$. 
We define a v-sheaf $\wt{J}_b$ on 
$\Perf_{\ol{\bF}_q}$ by 
\[
 \wt{J}_b (S) = \Aut (\sE_{b,S}) 
\]
for an $S \in \Perf_{\ol{\bF}_q}$. 
We note that 
the isomorphism class of 
$J_b$ and $\wt{J}_b$ depend only on 
$[b] \in B(G)$. 

For a locally profinite group $H$, we write 
$\ul{H}$ for v-sheaf on $\Perf_{\ol{\bF}_q}$ 
associated to $H$. 
Then we have an inclusion 
\[
 \ul{J_b(E)} \subset \wt{J}_b. 
\]
Let 
$\wt{J}_b^0$ be the connected component of 
the unit section of $\wt{J}_b$. 
Then we have 
\[
 \wt{J}_b = \wt{J}_b^0 \rtimes \ul{J_b(E)} 
\]
and $\wt{J}_b^0$ is of dimension $\langle 2\rho, \nu_{b} \rangle$ by \cite[Proposition III.5.1]{FaScGeomLLC}. In particular $\ul{J_b(E)} = \wt{J}_b$ 
if and only if $b$ is basic. 

Let $\Bun_G^{\mathrm{ss}}$ be the 
semi-stable locus of $\Bun_G$. 
Then $\Bun_G^{\mathrm{ss}}$ is an open substack of 
$\Bun_G$ by \cite[Theorem III.4.5]{FaScGeomLLC}]. 
Let $\alpha \in \pi_1 (G)_{\Gamma}$. Then the upper semicontinuity of $\mathrm{HN}$ provides a stratification 
\[
 \Bun_{G}^{\alpha} = 
 \coprod_{\nu \in X_*(A)^+_{\bQ}} 
 \Bun_{G}^{\alpha,\mathrm{HN}=\nu}.
\]
Take $\nu \in X_*(A)^+_{\bQ}$ and assume that 
$\Bun_{G}^{\alpha,\mathrm{HN}=\nu}$ 
is not empty. 
Then we have a unique 
$[b] \in B(G)$ such that 
$\kappa ([b])=\alpha$ and $\mathrm{HN}([b])=\nu$. 
Take any representative 
$b$ of $[b]$. 
Then by \cite[Proposition III.5.3]{FaScGeomLLC} we have an isomorphism 
\[
 x_b \colon [\Spa (\ol{\bF}_q )/\wt{J}_b ] 
 \stackrel{\sim}{\lra} 
 \Bun_{G}^{\alpha,\mathrm{HN}=\nu}  
\]
defined by $\sE_{b}$. 
If $b$ is basic, 
then $\Bun_{G}^{\alpha,\mathrm{HN}=\nu}$ 
is equal to the semi-stable locus 
$\Bun_{G}^{\alpha,\mathrm{ss}}$ of 
$\Bun_{G}^{\alpha}$ by \cite[Theorem III.4.5]{FaScGeomLLC}]. 

The $\wt{J}_b$-torsor 
$\sT_b$ over 
$\Bun_{G}^{\alpha,\mathrm{HN}=\nu}$ 
given by $x_b$ is the torsor defined by the functor 
which sends $S \in \Perf_{\ol{\bF}_q}$ to 
\[
 \Bigl( f \colon S \lra 
 \Bun_{G}^{\alpha,\mathrm{HN}=\nu}, 
 \phi \colon \sE_{b,S} \stackrel{\sim}{\lra} 
 \sE_f \Bigr) , 
\]
where $\sE_f$ is the $G$-bundle on $X_S$ 
determined by $f$, and 
$g \in \wt{J}_b (S)$ acts on 
$\sT_b (S)$ (on the right) by 
\begin{equation}\label{eq:ac}
 (f,\phi) \mapsto (f,\phi \circ g ). 
\end{equation}
Then we have 
$\Frob^* x_b =x_{\sigma(b)}$ and 
$\Frob^* \sT_b =\sT_{\sigma(b)}$. 
Since we have $\sigma (b)=b^{-1} b \sigma(b)$, 
we have a Weil descent datum 
\begin{equation}\label{eq:wb}
 w_b \colon \Frob^* \sT_b \lra \sT_b 
\end{equation}
induced by $t_{b^{-1}} \colon \sE_{b,S} \to \sE_{\sigma(b),S}$ 
in \eqref{eq:tg}. Explicitly at the level of $S$-points, \eqref{eq:wb} sends $(f,\phi)$ to $(f, \phi \circ t_{b^{-1}})$. If $b'=g b \sigma(g)^{-1}$, 
then 
$t_g^{-1}$ induces 
an isomorphism 
$\sT_b \to \sT_{b'}$, 
which is compatible with 
the Weil descent data $w_b$ and $w_{b'}$. 
Hence the isomorphism class of 
$(\sT_b,w_b)$ depends only on 
$[b] \in B(G)$.

\section{The global Hecke stack}\label{sec:Hsta}

Let $\Div^1_{X,\bF_q}$ be the moduli space of degree $1$ closed Cartier divisors 
defined in \cite[Definition II.1.19]{FaScGeomLLC}, 
which sends 
$S \in \Perf_{\bF_q}$ 
to the set of isomorphism classes of 
degree $1$ closed Cartier divisors on $X_S$. 
By \cite[Proposition II.1.21]{FaScGeomLLC}, 
$\Div_{X,\bF_q}^1 \to \Spa (\bF_q)$ is representable in spatial diamonds and  
we have an isomorphism 
\[
 \Spa (E)^{\diamond} 
 / \varphi_{E^{\diamond}}^{\bZ} 
 \stackrel{\sim}{\lra} \Div_{X,\bF_q}^1 , 
\] 
where $\varphi_{E^{\diamond}}$ is a $q$-th power 
Frobenius action on $E^{\diamond}$. 
We put $\Div_{X}^1 =\Div_{X,\bF_q}^1 \times_{\bF_q} \ol{\bF}_q$. 

We write 
$X_{*}(T)^+$ for the set of dominant 
cocharacters of $T$. 
Let $\mu \in X_*(T)^+/\Gamma$. 
We define a Hecke stack $\Hecke^{\leq \mu}$ 
as the fibered category in groupoids 
whose fiber at an affinoid perfectoid $\Spa (R,R^+) \in \Perf_{\bF_q}$ 
is the groupoid of quadruples $(\sE, \sE', D, f)$, 
where 
\begin{itemize}
 \item 
 $\sE$ and $\sE'$ are $G$-bundles on $X_R^{\sch}$, 
 \item 
 $D$ is an effective Cartier divisor of degree $1$ 
 on $X_R^{\sch}$ given by some untilt of $R$, 
 \item 
 the isomorphism 
 \[
  f \colon 
  \sE|_{X_R^{\sch} \setminus D} \stackrel{\sim}{\lra} 
  \sE'|_{X_R^{\sch} \setminus D}
 \]
 is a modification, which is bounded by $\mu$ 
 geometric fiberwisely. 
\end{itemize}
Then we have morphisms 
\[
 \xymatrix{
 & \Hecke^{\leq \mu} 
 \ar@{->}[ld]_-{\ola{h}} \ar@{->}[rd]^-{\ora{h}} & \\ 
 \Bun_G  & & 
 \Bun_G \times \Div_X^1 
 } 
\]
defined by 
$\ola{h} (\sE, \sE', D, f) =\sE'$ and 
$\ora{h} (\sE, \sE', D, f) =(\sE,D)$. 

In the sequel, a diamond means 
a diamond on $\Perf_{\ol{\bF}_q}$. 
Let $\ell$ be a prime number different from $p$. 
As we will need the natural functor (i.e. relative homology) constructed in \cite{FaScGeomLLC}, let us briefly review it. For $X$ a small v-stack, the derived category of solid $\ol{\bQ}_{\ell}$-sheaves $D_{\solid}( X, \ol{\bQ}_{\ell})$ is constructed in \cite[Definition VII.1.17]{FaScGeomLLC}. 
For what follows all tensor products are solid tensor products as constructed in \cite[Proposition VII.2.2]{FaScGeomLLC}.
For a map 
$f \colon X \to Y$ of small v-stacks, 
there is a functor 
\[
  f_{\natural} \colon 
 D_{\solid}( X, \ol{\bQ}_{\ell}) 
 \to D_{\solid}( Y, \ol{\bQ}_{\ell})  
\]
constructed in \cite[\S VII.3]{FaScGeomLLC}. 
See \cite[Proposition VII.3.1]{FaScGeomLLC} for basic properties of this functor. 
For an $\ell$-cohomologically smooth morphism $f \colon X \to Y$ of diamonds, we put 
\[
f^!\ol{\bQ}_{\ell}=\varprojlim_{n} Rf^! (\bZ/\ell^n \bZ) \otimes_{\bZ_{\ell}} \ol{\bQ}_{\ell} \in D_{\solid}( X, \ol{\bQ}_{\ell}). 
\]
For an Artin v-stack $X$, let 
$D_{\lis}(X,\ol{\bQ}_{\ell}) \subset D_{\solid}( X, \ol{\bQ}_{\ell})$ be the subcategory defined in 
\cite[Definition VII.6.1]{FaScGeomLLC}.

Let $\cD_{\infty}$ be a diamond over $\bC_p^{\flat}$ with an action of a profinite group $K_0$. 
Let $f_{\infty} \colon \cD_{\infty} \to \Spa (\bC_p^{\flat})$ be the structure morphism. 
Assume that the action of $K_0$ on geometric points of $\cD_{\infty}$ is free and the quotient diamond $\cD_{\infty}/K_0$ is an $\ell$-cohomologically smooth diamond over $\bC_p^{\flat}$. 
For an open subgroup $K$ of $K_0$, we put $\cD_K =\cD_{\infty}/K$, and let $f_K \colon \cD_K \to \Spa (\bC_p^{\flat})$ be the induced morphism. 
Then we put 
\[ H_{\mathrm{c}}^i (\cD_{\infty},\ol{\bQ}_{\ell})  = \varinjlim_{K \subset K_0}  R^i f_{K,\natural} ((f_K^!\ol{\bQ}_{\ell})^{\vee})
\]
for $i \geq 0$. Let $f \colon \cD \to \Spa (\bC_p^{\flat})$ be an $\ell$-cohomologically smooth morphism of diamonds.
For $\sF \in D_{\solid}( \cD, \ol{\bQ}_{\ell})$ and $i \geq 0$, we put 
\[
 H_{\mathrm{c}}^i (\cD,\sF ) = 
 R^i f_{\natural} (\cF \otimes (f^!\ol{\bQ}_{\ell})^{\vee}). 
\]
Let $h \colon \cM \to \cD$ be a $G_0$-torsor such that 
\[
 \varinjlim_{K \subset G_0} R f_{K,\natural} ((f_K^!\ol{\bQ}_{\ell})^{\vee}) \in D_{\lis}(\Spa (\bC_p^{\flat}),\ol{\bQ}_{\ell}) , 
\] 
where $G_0$ is a locally profinite group, $K$ runs along compact open subgroups of $G_0$ and $f_K \colon \cM /K \to \Spa (\bC_p^{\flat})$. 
Then we can regard $H_{\mathrm{c}}^j (\cM ,\ol{\bQ}_{\ell})$ as a smooth representation of $G_0$. 
Let $\pi$ be a smooth representation of 
$G_0$ 
over $\ol{\bQ}_{\ell}$. 
We 
define $\sF_{\pi} \in D_{\lis}( \cD, \ol{\bQ}_{\ell})$ 
as the pushforward of $\cM$ by $\pi$. 
Then we have a spectral sequence 
\begin{align}\label{eq:HSsp}
 H_i \bigl( G_0, H_{\mathrm{c}}^j (\cM ,\ol{\bQ}_{\ell}) \otimes \pi \bigr) 
 \Rightarrow H_{\mathrm{c}}^{j-i} (\cD,\sF_{\pi} ). 
\end{align}
This follows from \cite[Proposition VII.3.1]{FaScGeomLLC} as in the proof of \cite[Lemma 1.4]{ImaConv}.

\section{Fargues' conjecture}\label{sec:Fconj}
We recall the Hecke eigensheaf property in 
Fargues' conjecture in the 
case where 
the Langlands parameter is cuspidal and 
$\mu$ is minuscule. 
Up to some technicalities which were worked out in \cite{FaScGeomLLC}, we refer the reader to \cite[Conjecture 4.4(4)]{FarGover} for the general case. 

Let $\widehat{G}$ and ${}^L G$ be the dual group and L-group of $G$ over $\ol{\bQ}_{\ell}$. 
Let 
$\varphi \colon W_E \rightarrow {}^L G$ be 
a cuspidal $\ell$-adic L-parameter for $G$ (\cf \cite[Definition 1.15]{ImaLLCell}, \cite[Definition 4.1]{FarGover}). 
Let $S_{\varphi}$ be the centralizer of $\varphi$ in $\widehat{G}$. 
We fix a Whittaker datum. 
For $b \in B(G)_{\basic}$, 
let 
$\{ \pi_{\varphi,b,\rho} \}_{\rho \in \wh{S}_{\varphi}}$ 
be the $L$-packet 
corresponding to $\varphi$ 
by the local Langlands correspondence 
for the extended pure inner form $J_b$ of $G$ 
(\cf \cite[Conjecture 2.4.1]{KalSiso}). 
We recall that we have a decomposition 
\[
 \Bun_{G}^{\mathrm{ss}} = 
 \coprod_{\alpha \in \pi_1 (G)_{\Gamma}} 
 \Bun_{G}^{\alpha ,\mathrm{ss}} 
\]
into open and closed substacks. 
Let $\sF_{\varphi}$ be the object of 
$D_{\lis} (\Bun_{G},\ol{\bQ}_{\ell})$ 
with an action of $S_{\varphi}$ determined by the following conditions: 
\begin{itemize}
 \item 
 The support of $\sF_{\varphi}$ is contained in 
 $\Bun_{G}^{\mathrm{ss}}$. 
 \item 
 Let $\alpha \in \pi_1 (G)_{\Gamma}$.  
 Take a basic element $b \in G(\breve{E})$ 
 such that $\alpha=\kappa ([b])$. 
 Let 
 $\rho \in \wh{S}_{\varphi}$. 
 Let $\ul{\rho}$ be the 
 constant $\ol{\bQ}_{\ell}$-sheaf with action of 
 $S_{\varphi}$ on 
 $\Bun_{G}^{\alpha ,\mathrm{ss}}$  
 associated to $\rho$. 
 Let 
 $\ul{\pi_{\varphi,b,\rho}}$ be  
 the object of 
 $D_{\lis} (\Bun_{G}^{\alpha ,\mathrm{ss}},\ol{\bQ}_{\ell})$ 
 obtained as the pushforward of 
 the $\ul{J_b (E)}$-torsor $\sT_b$ under 
 $\pi_{\varphi,b,\rho}$. 
 Then we have 
 \begin{equation}\label{eq:resxb}
  \sF_{\varphi}|_{\Bun_{G}^{\alpha ,\mathrm{ss}}} 
  = 
  \bigoplus_{\rho \in \wh{S}_{\varphi},\, \rho|_{Z(\wh{G})^{\Gamma}} =\alpha} 
  \ul{\rho} \otimes \ul{\pi_{\varphi,b,\rho}}, 
 \end{equation}
 where we view $\alpha$ as an element of 
 $X^* (Z(\wh{G})^{\Gamma})$ 
 under the canonical isomorphism 
 $\pi_1 (G)_{\Gamma} \simeq X^* (Z(\wh{G})^{\Gamma})$. 
 The isomorphism class of 
 the right hand side of \eqref{eq:resxb} 
 does not depend on the choice of $b$, since the same is true for 
$\sT_b$. 
\end{itemize}
Then properties (1), (2) and (3) of \cite[Conjecture 4.4]{FarGover} are immediate. 

Take a representative $\mu' \in X_* (T)^+$ 
of $\mu$. 
Let $\Gamma'$ be the stabilizer of $\mu'$ in $\Gamma$. 
We put 
\[
 r_{\mu} = \Ind_{\wh{G} \rtimes \Gamma'}^{{}^L G} 
 r_{\mu'} , 
\]
where 
$r_{\mu'}$ is the highest weight $\mu'$ 
irreducible representation of 
$\wh{G} \rtimes \Gamma'$. 

As in \cite[IX.2]{FaScGeomLLC}, we can construct a functor 
\begin{equation}\label{eq:funcSatshf}
	\Rep_{\ol{\bQ}_{\ell}} ({}^L G) \to 
	D_{\solid} (\mathrm{Hecke}^{\leq \mu},\ol{\bQ}_{\ell} );\ V \mapsto \cS'_V 
\end{equation}
via the geometric Satake equivalence (\cf \cite[\S 10]{ImaGeomLLC}). 
Let $\IC_{\mu}'$ be the image of $r_{\mu}$ under  the functor \eqref{eq:funcSatshf}.

Now we can state the Hecke eigensheaf property in 
Fargues' conjecture: 

\begin{conj}
We have $$\overrightarrow{h}_{\natural}(\overleftarrow{h}^{*} 
 \sF_{\varphi} {\otimes}_{\overline{\bQ}_{\ell}} \IC_{\mu}') \cong \sF_{\varphi} \boxtimes (r_{\mu} \circ \varphi )$$
as objects of $D_{\solid} (\Bun_{G} \times \Div_{X}^1, \ol{\bQ}_{\ell})$ 
with actions of $S_{\varphi}$. 
\end{conj}

\section{Non-semi-stable locus}\label{sec:vcoh}
Let $b ,b' \in G(\breve{E})$. 
We have a natural morphism 
\[
 y_b \colon 
 [\Div_{X}^1 /\wt{J}_b ] 
 \simeq 
 [\Spa ( \ol{\bF}_q ) /\wt{J}_b ] \times 
 \Div_X^1 
 \xra{(x_b,\id )} 
 \Bun_{G} \times \Div_X^1 . 
\]
We consider the cartesian diagram (i.e. every sub-square is cartesian) 
\begin{equation*}
 \xymatrix{
 \Hecke_{b,b'}^{\leq \mu} \ar@{->}[r] \ar@{->}[dd]_-{\ola{h}_{b,b'}} 
 & 
 \Hecke_b^{\leq \mu} \ar@{->}[r] \ar@{->}[d] & 
 [\Div_{X}^1 /\wt{J}_b ] 
 \ar@{->}^-{y_b}[d] \\
  & 
 \Hecke^{\leq \mu} \ar@{->}^-{\overrightarrow{h}}[r] 
 \ar@{->}^-{\overleftarrow{h}}[d] & 
 \Bun_{G} \times \Div_X^1 \\ 
 [\Spa (\ol{\bF}_q)/\wt{J}_{b'} ] \ar@{->}[r]^-{x_{b'}} & 
 \Bun_{G}. &  
 }
\end{equation*}

By the construction, 
for a perfectoid affinoid $\ol{\bF}_q$-algebra $(R,R^+)$, 
the groupoid 
$\Hecke_{b,b'}^{\leq \mu} (R,R^+)$
consists of quadruples 
$(\sE, \sE', D, f)$, 
where 
\begin{itemize}
 \item 
 $\sE$ and $\sE'$ are $G$-bundles on $X_R^{\mathrm{sch}}$ 
 which are isomorphic to $\sE_b$ and $\sE_{b'}$ 
 fiberwisely over $\Spa (R,R^+)$. 
 \item 
 $D$ is an effective Cartier divisor of degree $1$ 
 on $X_R^{\sch}$, 
 \item 
 $f \colon \sE|_{X_R^{\mathrm{sch}} \setminus D} \to \sE'|_{X_R^{\mathrm{sch}} \setminus D}$ is a modification 
 bounded by $\mu$ geometric fiberwisely over $\Spa (R,R^+)$. 
\end{itemize}
Let $\cT_{b,b'}^{\leq \mu}$ be 
the $\wt{J}_b$-torsor over $\Hecke_{b,b'}^{\leq \mu}$ 
obtained by considering an isomorphism 
$\phi \colon \sE_b \xrightarrow{\sim} \sE$. 
Let 
$\Gr_{b,b'}^{\leq \mu}$ and 
$\cM_{b,b'}^{\leq \mu}$ be the 
$\wt{J}_{b'}$-torsors over 
$\Hecke_{b,b'}^{\leq \mu}$ and 
$\cT_{b,b'}^{\leq \mu}$ 
obtained by considering an isomorphism 
$\phi' \colon \sE_{b'} \xrightarrow{\sim} \sE'$ 
respectively. 
Then $\cM_{b,b'}^{\leq \mu}$ is a 
$\wt{J}_{b'}$-equivariant 
$\wt{J}_b$-torsor over $\Gr_{b,b'}^{\leq \mu}$. 
We have commutative diagrams 
\begin{equation*}
 \xymatrix{
 \cM_{b,b'}^{\leq \mu} \ar@{->}[r] \ar@{->}[d] & 
 \cT_{b,b'}^{\leq \mu} \ar@{->}[r] \ar@{->}[d] & 
 \Spa ( \breve{E} )^{\diamond} \ar@{->}[d] \\
 \Gr_{b,b'}^{\leq \mu} \ar@{->}[r] & 
 \Hecke_{b,b'}^{\leq \mu} \ar@{->}[r] & 
 [\Div_{X}^1 /\wt{J}_b ] , 
 }
\end{equation*}
where the sub-squares are cartesian. 

By \cite[Proposition 3.20]{FarGover}, $\cT_{b,b'}^{\leq \mu}$ is a diamond. 
Furthermore by \cite[Lemma 10.13, Proposition 11.5]{SchEtdia}, $\cM_{b,b'}^{\leq \mu}$ 
is a diamond if $b'$ is basic. 

\begin{rem}
The maps $\cM_{b,b'}^{\leq \mu} \to \Gr_{b,b'}^{\leq \mu}$ and $\cM_{b,b'}^{\leq \mu} \to \cT_{b,b'}^{\leq \mu}$ appearing in the above diagram are generalized versions of the Hodge--Tate period map and the Gross--Hopkins period map. Indeed if $b'=1$ and $\mu$ is minuscule then $\cM_{b,b'}^{\leq \mu} \to \Gr_{b,b'}^{\leq \mu}$ is the usual Hodge--Tate period map of a Rapoport--Zink space at infinite level associated to the isocrystal $b$ and $\cM_{b,b'}^{\leq \mu} \to \cT_{b,b'}^{\leq \mu}$ is the usual Gross--Hopkins period map. On the other hand if $b=1$ and $\mu$ is minuscule then $\cM_{b,b'}^{\leq \mu} \to \Gr_{b,b'}^{\leq \mu}$ is the Gross--Hopkins map and  $\cM_{b,b'}^{\leq \mu} \to \cT_{b,b'}^{\leq \mu}$ is the Hodge--Tate map associated to the isocrystal $b'$. 
\end{rem}

For a finite dimensional algebraic 
representation $V$ of $G$ and a rational number $\alpha$, 
we put 
\[
 \Fil_b^{\alpha} V = \bigoplus_{\alpha' \leq - \alpha} V_{\alpha'} , 
\]
where 
\[
 V = \bigoplus_{\alpha \in \bQ} V_{\alpha} 
\]
is the slope decomposition given by 
$\nu_b \in X_* (A)_{\bQ}^+$. 
This gives a filtration 
$\Fil_b$ 
on the forgetful fiber functor 
$\omega \colon \Rep (G) \to \mathrm{Vect}_E$ 
(\cf \cite[IV, 2.1]{SaaCatT}). 
The stabilizer of 
$\Fil_b \omega$ gives a parabolic subgroup $P^b$ of $G$. 
Let $L^b$ 
be the centralizer of $\nu_b \in X_* (A)_{\bQ}^+$. 
Take a Levi subgroup $L$ of $G$ containing $L^b$. 
We put $P=L P^b$. 
Then, 
$P$ is a parabolic subgroup of $G$ and 
$[b] \in B(G)$ is the image of 
an element $b_{00} \in L^b (\breve{E})$. 
Let $b_0$ be the image of $b_{00}$ in $L (\breve{E})$. 

We take a cocharacter 
$\lambda \in X_* (A)$ so that 
$P$ is associated to $\lambda$ 
in the sense of \cite[13.4.1]{SprLAG}. 
Then we have a 
filtration $\Fil_{\lambda}$ on $\omega$ 
associated to $\lambda$. 

We assume that $[b']$ is in the image of 
$B(L) \to B(G)$. 
Then $\Fil_{\lambda} \omega$ 
induces the filtrations 
$\Fil_{\lambda} \sE_b$ and $\Fil_{\lambda} \sE_{b'}$ as fiber functors 
by the construction, 
because $[b], [b']$ are in the image of 
$B(L) \to B(G)$ and 
$L$ is 
the centralizer of $\lambda$ in $G$. 

We define a closed subspace 
$\cC_{b,b'}^{\leq \mu}$ 
of $\Gr_{b,b'}^{\leq \mu}$ as a functor 
that sends 
a perfectoid affinoid $\ol{\bF}_q$-algebra 
$(R,R^+)$ 
to the isomorphism classes of 
$(\sE, \sE', D, f, \phi')$, 
where 
\begin{itemize}
 \item 
 $(\sE, \sE', D, f)$ is as in 
 $\Hecke_{b,b'}^{\leq \mu} (R,R^+)$, 
 \item 
 $\phi' \colon \sE_{b'} \xrightarrow{\sim} \sE'$ 
 and $f$ are compatible with 
 $\Fil_{\lambda} \sE_b$ and $\Fil_{\lambda} \sE_{b'}$ 
 geometric fiberwisely 
 in the sense that following holds for 
 any geometric point $\Spa (F,F^+)$ of $\Spa (R,R^+)$: 
 Take an isomorphism 
 $\sE_{b} \xrightarrow{\sim} \sE$ over $X_F^{\mathrm{sch}}$. 
 Let $D_F$ be a Cartier divisor of $X_F^{\mathrm{sch}}$  determined by $D$. 
 Then the composite 
 \[
  \sE_{b}|_{X_F^{\mathrm{sch}}\setminus D_F}  \stackrel{\sim}{\lra} 
  \sE|_{X_F^{\mathrm{sch}}\setminus D_F} 
  \stackrel{f}{\lra} 
  \sE'|_{X_F^{\mathrm{sch}}\setminus D_F} 
  \xra{\phi'^{-1}} 
  \sE_{b'}|_{X_F^{\mathrm{sch}}\setminus D_F}
 \] 
 respects the filtrations 
 $\Fil_{\lambda} \sE_b |_{X_F^{\mathrm{sch}}\setminus D_F}$ and $\Fil_{\lambda} \sE_{b'} |_{X_F^{\mathrm{sch}}\setminus D_F}$.
\end{itemize}

\begin{rem}
The condition that 
$\phi'$ and $f$ are compatible with 
$\Fil_{\lambda} \sE_b$ and $\Fil_{\lambda} \sE_{b'}$ 
is independent of choice of 
an isomorphism $\sE_{b} \xrightarrow{\sim} \sE$, 
because the automorphism group $\wt{J}_b$ of 
$\sE_{b}$ respects the filtration 
$\Fil_{\lambda} \sE_b$. 
\end{rem}

For $\mu \in X_* (T)$, 
we put 
\[
 \ol{\mu} = 
 \frac{1}{[\Gamma :\Gamma_{\mu}]} 
 \sum_{\tau \in \Gamma / \Gamma_{\mu}} 
 \tau (\mu) , 
\]
where $\Gamma_{\mu}$ is a stabilizer of $\mu$ 
in $\Gamma$, 
and let $\mu^{\natural}$ denote 
the image of $\mu$ in $\pi_1 (G)_{\Gamma}$. 

\begin{defn}(\cf \cite[Definition 2.5]{RVlocSh})
We say that 
$[b] \in B(G)$ is acceptable for 
$(\mu,[b'])$ 
if $\nu_b -\nu_{b'} \leq \ol{\mu}$. 
We say that 
$[b] \in B(G)$ is neutral for 
$(\mu,[b'])$ 
if $\kappa_G ([b]) -\kappa_G ([b']) = \mu^{\natural}$.
\end{defn}

Let $B(G,\mu,[b'])$ be the set of 
acceptable neutral elements in $B(G)$ 
for $(\mu,[b'])$.

\begin{rem}
The set $B(G,\mu,[b'])$ is a twisted analogue of the set $B(G,\mu)$, the latter due to Kottwitz. We refer the reader to \cite[\S 6.2]{KotIsoII} for this definition. 
\end{rem}
To state our main results we need the notion of Hodge--Newton reducibility.

\begin{defn}(\cf \cite[Definition 4.28]{RVlocSh})
A triple $([b],[b'],\mu)$ such that 
$[b] \in B(G,\mu,[b'])$ and $b'$ is basic 
is called Hodge--Newton reducible, 
if there is a standard proper 
Levi subgroup $L$ of $G$ and 
$[b_0] ,[b_0'] \in B(L)$ 
such that
$[b]$ and $[b']$ are the images of $[b_0]$ 
and $[b_0']$ respectively, 
$\mu$ factors through $L$, 
$[b_0] \in B(L,\mu,[b'_0])$ and 
the action of $\nu_{b_0}$ on $R_{\mathrm{u}}(B)$ is non-negative. 
\end{defn}

\begin{lem}\label{lem:DVRquot}
Let $R$ be a DVR with the maximal ideal $\fm$, and 
$M$ be an $R$-module such that 
$M \simeq \bigoplus_{1 \leq i \leq n} R/\fm^{k_i}$, 
where $k_1 \geq \cdots \geq k_n$ 
is a sequence of non-negative integers. 
Let $N$ be a quotient of 
$M$ generated by $j$ elements, 
where $j \leq n$. 
Then we have 
$l (N) \leq k_1 + \cdots + k_j$. 
Further, if the equality holds, 
then $N$ is a direct summand of $M$. 
\end{lem}
\begin{proof}
This follows from \cite[Lemma 3.2]{HanHarr} 
by taking the Pontryagin dual. 
\end{proof}

The following proposition is a slight generalization of 
\cite[Theorem 3.1]{HanHarr}, 
where the slope of a semi-stable bundle 
is assumed to be zero. 

\begin{prop}\label{prop:vecss}
Assume that $G=\GL_n$. 
Let $(k_1 \geq \cdots \geq k_n)$ 
be the sequence of integers 
corresponding to $\mu \in X_*(T)^+$. 
Let $(R,R^+)$ be a perfectoid affinoid $\ol{\bF}_q$-algebra. 
Let 
\[
  f \colon 
  \sE|_{X_R^{\sch} \setminus D} \stackrel{\sim}{\lra} 
  \sE'|_{X_R^{\sch} \setminus D}
\]
be a modification of between $G$-bundles 
$\sE$ and $\sE'$ over $X_R^{\sch}$ along an 
effective Cartier divisor of degree $1$ 
which is equal to 
$\mu$ geometric fiberwisely. 
We view $\sE$ and $\sE'$ 
as vector bundles of rank $n$. 
Let $\sE^+$ be a saturated sub-vector bundle of $\sE$ 
such that 
\begin{equation}\label{eq:degEk}
 \deg (\sE_x^+ ) + 
 \sum_{1 \leq j \leq \rk (\sE^+)} 
 k_{n+1-j} = \rk (\sE^+) s  
\end{equation}
for every point $x$ of $Spa (R,R^+)$. 

Assume that 
$\sE'$ is semi-stable of slope $s$ geometric fiberwisely. 
Let 
$j \colon X_R^{\sch} \setminus D \to X_R^{\sch}$ 
be the open immersion. 
We put 
\[
 \sE'^+ =j_* f(j^*\sE^+ ) \cap \sE' . 
\]
Then 
$\sE'^+$ is a semi-stable 
vector bundle of slope $s$ such that 
$\rk (\sE'^+) =\rk (\sE^+)$. 
\end{prop}
\begin{proof}
We follow arguments in the proof of 
\cite[Theorem 3.1]{HanHarr}. 

Take a modification 
$f_1 \colon \cO|_{X_R^{\sch} \setminus D} \xra{\sim} \cO (1)|_{X_R^{\sch} \setminus D}$ of degree $1$ along $D$. 
For a large $N$, changing 
$\sE'$, $f$ and 
$(k_1 ,\ldots ,k_n)$ by 
$\sE' (N)$, 
\[
 (\id_{\sE'} \otimes f_1^{\otimes N}) \circ f \colon 
 \sE |_{X_R^{\sch} \setminus D} \stackrel{\sim}{\lra} 
 \sE'(N)|_{X_R^{\sch} \setminus D} 
\]
and 
$(k_1 +N ,\ldots ,k_n +N )$ respectively, 
we may assume that $f$ 
extends to an injective morphism 
$f \colon \sE \to \sE'$, 
which induces a morphism 
$f^+ \colon \sE^+ \to \sE'^+$. 
We put 
$\sE^- =\sE / \sE^+$ and 
$\sE'^- =\sE' / \sE'^+$. 
Let $f^- \colon \sE^- \to \sE'^-$ 
be the morphism induced by $f$. 

First, we treat the case where 
$R$ is a perfectoid field. 
In this case, 
$\sE'^+$ and $\sE'^-$ are vector bundles such that 
$\rk (\sE'^+) = \rk (\sE^+)$ and 
$\rk (\sE'^-) = \rk (\sE^-)$.  
Let $Q^+$ and $Q^-$ be the cokernel of 
$f^+$ and $f^-$ respectively. 
Then we have 
\[
 l (Q^-) \leq \sum_{1 \leq i \leq \rk (\sE^-) } 
 k_i 
\] 
by Lemma \ref{lem:DVRquot}, 
since $Q^-$ is generated by $\rk (\sE^-)$-elements. 
Hence we have 
\[
 l (Q^+) \geq \sum_{1 \leq j \leq \rk (\sE^+) } 
 k_{n+1-j} . 
\]
By this and \eqref{eq:degEk}, 
we have 
\[
 \deg (\sE'^+ )= 
 \deg (\sE^+) +l (Q^+) \geq 
 \rk (\sE^+) s . 
\]
On the other hand, we have 
$\deg (\sE'^+ ) \leq \rk (\sE^+) s$, 
since $\sE'$ is semi-stable. 
Therefore, 
$\sE'^+$ is a semi-stable vector bundle of 
slope $s$. 

The general case is reduced to the above case 
by the same argument as in \cite[\S 3.2]{HanHarr}. 
\end{proof}

\begin{lem}\label{lem:HOvan}
Let $(R,R^+)$ be a perfectoid affinoid $\ol{\bF}_q$-algebra. 
For any element $\alpha$ of 
$H_{\et}^1 (X_R^{\sch},\cO )$, 
there is a pro-etale extension 
$(R',R'^+)$ of $(R,R^+)$ such that 
the image of $\alpha$ in $H_{\et}^1 (X_{R'}^{\sch},\cO )$ 
is zero. 
\end{lem}
\begin{proof}
Any extension of $\cO$ by 
$\cO$ on $X_R^{\sch}$ 
splits after 
a pro-etale extension of 
$(R,R^+)$ 
by \cite[6.3.1]{FaFoVbp} and 
\cite[Theorem 2.26]{FarGover} (\cf \cite[Corollary 8.7.10]{KeLiRpHF} ). 
This implies the claim, 
since 
$H_{\et}^1 (X_R^{\sch},\cO )$ 
parametrize the extensions of $\cO$ by 
$\cO$ on $X_R^{\sch}$. 
\end{proof}

Assume that $b'$ is basic. 
Let $U$ be the unipotent radical of $P$. 
Note that we have a surjection 
\[
 P \lra P/U \simeq L, 
\]
where the second isomorphism is given by 
$L \hookrightarrow P \to P/U$. 

\begin{lem}\label{lem:EPredL}
Let $(R,R^+)$ be a perfectoid affinoid $\ol{\bF}_q$-algebra. 
Let $\sE_P$ a $P$-bundle on 
$X_R^{\sch}$ such that 
$\sE_P \times^P L \simeq \sE_{b'_0}$. 
Then we have an isomorphism 
$\sE_P \simeq \sE_{b'_0} \times^L P $ 
after 
a pro-etale extension of 
$(R,R^+)$. 
\end{lem}
\begin{proof}
We follow arguments in the proof of 
\cite[Proposition 5.16]{FarGtor}. 
Let $P$ act on $U$ by the conjugation. 
We put 
\[
 \sU = \sE_P \times^P U . 
\]
Then 
$H_{\et}^1 (X_R^{\sch},\sU)$ parametrizes 
the fiber of 
\[
 H_{\et}^1 (X_R^{\sch}, P) \lra H_{\et}^1 (X_R^{\sch},L) 
\]
over the image of $\sE_P$. 
Hence, it suffices to show that 
$H_{\et}^1 (X_R^{\sch},\sU)$ is trivial 
after 
a pro-etale extension of 
$(R,R^+)$. 
This follows from Lemma \ref{lem:HOvan}, 
since 
$\sU$ has a filtration whose 
graded subquotients are 
semi-stable vector bundles of slope zero. 
\end{proof}

\begin{lem} \label{lem:clHec}
Let $\mu_1,\mu_2 \in X_{*}(T)^{+}$ such that $\mu_1 \leq \mu_2$. 
Then $\Hecke^{\leq \mu_1} \subset \Hecke^{\leq \mu_2}$ 
is a closed substack. 
\end{lem}

\begin{proof}
By \cite[Proposition 3.20]{FarGover}, it 
is enough to prove 
$\Gr_G^{\leq \mu_1} \subset \Gr_G^{\leq \mu_2}$ 
is closed substack. 
The latter follows from 
the semi-continuity of the map 
$\lvert \Gr \rvert \to X_* (T)^+/\Gamma$ 
in \cite[3.3.2]{FarGover} 
(\cf \cite[Proposition 19.2.3]{ScWeBLp}). 
\end{proof}

We define a substack 
$\Hecke^{\mu}$ of 
$\Hecke^{\leq \mu}$ 
by requiring the condition that 
modifications are equal to $\mu$ 
geometric fiberwisely. 
Then $\Hecke^{\mu}$ is an open substack of 
$\Hecke^{\leq \mu}$ by Lemma \ref{lem:clHec}. 
We use similar definitions and notations 
also for other spaces. 

Let $X$ be a scheme over $E$. 
Let $\mathrm{FilVect}_{X}$ 
be the category of filtered vector bundles on 
$X$. 
We consider the functor 
\[
 \omega_{\lambda} \colon \Rep (G) \lra \mathrm{FilVect}_{X} ;\ 
 V \mapsto (V \otimes_E \cO_X , (\Fil_{\lambda} V) \otimes_E \cO_X) . 
\] 
Let $\mathrm{Fil}_{\lambda}\mathrm{Bun}^{G}_{X}$ 
be the category of 
functors 
$\omega \colon \Rep (G) \to \mathrm{FilVect}_{X}$ 
which are isomorphic to 
$\omega_{\lambda}$ fpqc locally on $X$.  
Let $\mathrm{Bun}^P_{X}$ 
be the category of 
$P$-bundles on $X$. 

\begin{lem}\label{lem:Pfil}
There is an equivalence of categories 
\[
 \mathrm{Fil}_{\lambda}\mathrm{Bun}^{G}_{X} \lra 
 \mathrm{Bun}^P_{X} ;\ 
 \omega \mapsto \ul{\Isom}^{\otimes}_{X} (\omega_{\lambda},\omega), 
\]
where 
$\ul{\Isom}^{\otimes}_{X} (\omega_{\lambda},\omega)$ 
is a functor from the category of schemes over $X$ 
to the category of sets 
which sends $X'$ to the set of isomorphisms 
$\omega_{\lambda}|_{X'} \to \omega|_{X'}$ 
as filtered tensor functors. 
\end{lem}
\begin{proof}
This follows from \cite[Theorem 4.42 and Theorem 4.43]{ZieGrfilT}. 
\end{proof}

\begin{prop}\label{prop:redPf}
Assume that 
$([b],[b'],\mu)$ is Hodge--Newton reducible 
for $L$. 
Let $(R,R^+)$ be a perfectoid affinoid $\ol{\bF}_q$-algebra, and 
$(\sE, \sE', D, f) \in \Hecke_{b,b'}^{\mu} (R,R^+)$. 
Then, after taking a pro-etale extension of 
$(R,R^+)$, 
there is a reduction 
\[
 f_P \colon 
 \sE_P|_{X_R^{\sch} \setminus D} \stackrel{\sim}{\lra} 
 \sE'_P|_{X_R^{\sch} \setminus D}
\]
of $f$ to $P$ such that 
$\sE_P \simeq \sE_{b_0} \times^L P$ 
and 
$\sE'_P \simeq \sE_{b_0'} \times^L P$. 
\end{prop}
\begin{proof}
By taking a pro-etale extension of 
$(R,R^+)$, 
we can take an isomorphism 
$\sE_b  \simeq  \sE $. 
We put $\sE_P =\sE_{b_0} \times^L P$. 
Then $\sE_P$ and the isomorphism 
\[
 \sE_P \times^P G \cong 
 \sE_{b_0} \times^L G \cong \sE_b 
 \stackrel{\sim}{\lra} 
 \sE  
\]
give a reduction of $\sE$ to $P$. 
We put $\phi_P =\id_{\sE_{b_0} \times^L P}$. 
Then $\phi_P$ is a reduction of $\phi$ to $P$. 

For any irreducible $V \in \Rep (G)$, 
the vector bundle $\sE'(V)$ is semi-stable 
geometric fiberwisely. 
By Proposition \ref{prop:vecss}, 
we have a functorial construction of 
a filtration of $\sE'(V)$ 
that is compatible under 
$f(V)$ with the filtration 
of $\sE(V)$ coming from $\sE_P$ by Lemma \ref{lem:Pfil}. 
Since the category $\Rep (G)$ is semi-simple, 
the construction extends to 
all $V \in \Rep (G)$ in a functorial way. 
Hence, by Lemma \ref{lem:Pfil}, we have a reduction  
\[
 f_P \colon 
 \sE_P|_{X_R^{\sch} \setminus D} \stackrel{\sim}{\lra} 
 \sE'_P|_{X_R^{\sch} \setminus D}
\]
of $f$ to $P$ for some $P$-bundle $\sE'_P$. 
By Lemma \ref{lem:EPredL}, 
$\sE'_P$ is isomorphic to $\sE_{b_0'} \times^L P$ 
after taking a pro-etale extension of 
$(R,R^+)$. 
\end{proof}

Let $\wt{P}_{b'}$ be the 
stabilizer of $\Fil_{\lambda} \sE_{b'}$ 
in $\wt{J}_{b'}$. 
Then $\wt{P}_{b'} =\ul{P_{b'} (E)}$ 
for a parabolic subgroup $P_{b'}$ of 
$J_{b'}$.

\begin{prop}\label{prop:SPJ}
Assume that 
$([b],[b'],\mu)$ is Hodge--Newton reducible for $L$. 
Then the action of 
$\wt{P}_{b'}$ on $\Gr_{b,b'}^{\mu}$ 
stabilizes 
$\cC_{b,b'}^{\mu}$, 
and we have a natural 
$\wt{J}_{b'}$-equivariant isomorphism 
\[
 \cC_{b,b'}^{\mu} 
 \times^{\wt{P}_{b'}} \wt{J}_{b'} 
 \stackrel{\sim}{\lra} 
 \Gr_{b,b'}^{\mu} . 
\]
\end{prop}
\begin{proof}
The first claim follows from the 
definitions of 
$\wt{P}_{b'}$ and $\Gr_{b,b'}^{\mu}$. 
The morphism 
\[
 \cC_{b,b'}^{\mu} 
 \times^{\wt{P}_{b'}} \wt{J}_{b'} 
 \lra 
 \Gr_{b,b'}^{\mu} 
\]
induced by the action of $\wt{J}_{b'}$ on 
$\Gr_{b,b'}^{\mu}$ is an epimorphism 
by Proposition \ref{prop:redPf}. 

We show the injectivity. 
Let $g \in \wt{J}_{b'} (R,R^+)$ for  
a perfectoid affinoid $\ol{\bF}_q$-algebra 
$(R,R^+)$. 
Assume that 
$g$ sends a point of 
$\cC_{b,b'}^{\mu} (R,R^+)$ 
to a point of $\cC_{b,b'}^{\mu} (R,R^+)$. 
Then 
$g$ stabilizes $\Fil_{\lambda} \sE_{b'}$ 
outside the Cartier divisor corresponding to $R^{\sharp}$. 
This implies 
$g$ stabilizes 
$\Fil_{\lambda} \sE_{b'}$ on $X_R^{\mathrm{sch}}$, 
since $g$ stabilizes $\sE_{b'}$ itself. 
Hence, we have $g \in \wt{P}_{b'} (R,R^+)$. 
\end{proof}

Let 
$\cP_{b,b'}^{\mu}$ 
be the inverse image of 
$\cC_{b,b'}^{\mu}$ 
under 
$\cM_{b,b'}^{\mu} \to \Gr_{b,b'}^{\mu}$. 

\begin{cor}\label{cor:MPJ}
Assume that 
$([b],[b'],\mu)$ is Hodge--Newton reducible for $L$. 
Then the action of 
$\wt{P}_{b'}$ on $\cM_{b,b'}^{\mu}$ 
stabilizes 
$\cP_{b,b'}^{\mu}$, 
and we have a natural 
$( \wt{J}_b \times \wt{J}_{b'})$-equivariant 
isomorphism 
\[
 \cP_{b,b'}^{\mu} 
 \times^{\wt{P}_{b'}} \wt{J}_{b'} 
 \stackrel{\sim}{\lra} 
 \cM_{b,b'}^{\mu} . 
\]
\end{cor}
\begin{proof}
This follows from 
Proposition \ref{prop:SPJ}. 
\end{proof}

We define a subsheaf $\wt{J}_{b}^{U}$ of 
$\wt{J}_b$ by 
\[
 \wt{J}_{b}^{U} (S) = 
 \Bigl\{ g \in \wt{J}_b (S) \Bigm| g|_{\Fil_{\lambda}^j 
 \sE_b} \equiv \id_{\Fil_{\lambda}^j \sE_b} \mod \Fil_{\lambda}^{j+1} \sE_b \ 
 \textrm{for all $j$} \Bigr\} 
\]
for $S \in \Perf_{\ol{\bF}_q}$. 

Let $U_{b'}$ be the unipotent radical of $P_{b'}$. 
The inner form of $L$ determined by $b'$ gives 
a Levi subgroup $L_{b'}$ of $P_{b'}$. 

We use a notation that 
\[
 \gr_{\lambda}^i =\Fil_{\lambda}^i / \Fil_{\lambda}^{i+1} 
\]
for any integer $i$. 
Let $\rho_U$ be the half-sum of the positive roots $\alpha$ of $T$ such that $-\alpha$ occurs in the adjoint action of $T$ on $\Lie(U)$. 
We put 
$N_{U,b} = \langle 2 \rho_U , \nu_b \rangle$. 

\begin{defn} \label{defn:lcontractible}
Let $F$ be a non-archimedean field with a
valuation subring $F^+$. 
Let $f \colon D \to \Spa(F,F^+)^{\diamond}$ be 
an $\ell$-cohomologically smooth morphism of locally spatial diamonds 
(\cf \cite[Definition 23.8]{SchEtdia}). 
We say that $D$ is $\ell$-contractible of pure dimension $d$ 
if $f^! \bF_{\ell} = \bF_{\ell} (d)[2d]$ and the trace morphism 
$Rf_! f^! \bF_{\ell} \to \bF_{\ell}$ is a quasi-isomorphism. 
\end{defn}

\begin{rem}
In the situation of Definition \ref{defn:lcontractible}, we have $f_{\natural}\bF_{\ell} \cong Rf_{!}f^{!}\bF_{\ell}$ by \cite[Proposition VII.5.2]{FaScGeomLLC}. 
\end{rem}

Let $\varpi$ be a uniformizer of $E$. Let $\bB$ denote the v-sheaf on $\Perf_{\bF_{q}}$ given by $\bB(S) = \cO(Y_{S})$ (\cf \cite[Proposition II.2.1]{FaScGeomLLC}). 

\begin{lem} \label{lem:Bcontr}
Let $d$ and $h$ be positive integers. 
Let $f_{d,h} \colon \bB^{\varphi^d =\varpi^h} \times \Spa({\breve{E}})^{\diamond} \to \Spa({\breve{E}})^{\diamond}$ be the natural morphism. 
\begin{enumerate}
\item\label{en:Bcontr}
The v-sheaf $\bB^{\varphi^d =\varpi^h} \times \Spa({\breve{E}})^{\diamond}$ is 
an $\ell$-cohomologically smooth $\ell$-contractible locally spatial diamond 
of pure dimension $h$ over $\Spa({\breve{E}})^{\diamond}$. 
\item\label{en:BactE} 
The action of 
$E^{\times}$ on $f_{d,h,!} \bZ_{\ell}$ is given by $|| \cdot ||^{-d}$. 
\item\label{en:Badd}
Let $F$ be a perfectoid field over $\breve{E}$ and 
$a \in \bB^{\varphi^d =\varpi^h}(F^{\flat})$. 
Let $f_{d,h,F^{\flat}} \colon \bB^{\varphi^d =\varpi^h} \times \Spa(F^{\flat}) \to \Spa(F^{\flat})$ denote the base change of $f_{d,h}$. 
Then the action of $a$ on $f_{d,h,F^{\flat},!} \bZ_{\ell}$ induced by the addition on 
$\bB^{\varphi^d =\varpi^h}$ is trivial. 
\end{enumerate}
\end{lem}
\begin{proof}
Replacing $E$ by the unramified extension of degree $d$, we may assume that $d=1$ 
(\cf \cite[Remarque 4.2.2]{FaFoCfv}). 
We proceed by induction on $h \geq 1$. For $h=1$, 
the diamond $\bB^{\varphi = \varpi} \times \Spa({\breve{E}})^{\diamond}$ 
is isomorphic to 
$\Spa (\bF_q [[x^{1/p^{\infty}}]]) \times \Spa({\breve{E}})^{\diamond}$ by \cite[1.5.3]{FarGover}. 
The action of $\varpi$ on $\Spa (\bF_q [[x^{1/p^{\infty}}]]) \times \Spa({\breve{E}})^{\diamond}$ is induced from the morphism 
\[
 \Spa (\bF_q [[x^{1/q^m}]]) \to \Spa (\bF_q [[x^{1/q^m}]]) ;\ 
 x^{1/q^m} \mapsto x^{1/q^{m-1}} 
\]
of degree $q$ by taking limit with respect to $m \geq 0$. 
On the other hand, the action of $\cO_E^{\times}$ on 
$\Spa (\bF_q [[x^{1/p^{\infty}}]]) \times \Spa({\breve{E}})^{\diamond}$ is induced from an isomorphism on $\Spa (\bF_q [[x^{1/q^m}]])$ 
by taking limit with respect to $m \geq 0$. 
Further the addition of $a \in \Spa (\bF_q [[x^{1/p^{\infty}}]])(F^{\flat})$ 
on $\Spa (F^{\flat} [[x^{1/p^{\infty}}]])$ is 
induced from 
an isomorphism on $\Spa (\bF_q [[x^{1/q^m}]])$ 
by taking limit with respect to $m \geq 0$. 
Hence the claims hold for $h=1$ by \cite[Lemma 1.3]{ImaConv}. 

Assume that the result is true for $\bB^{\varphi = \varpi^{h-1}}$. 
We have an exact sequence 
\begin{equation}\label{eq:BAext}
    0 \lra \bB^{\varphi = \varpi^{h-1}} \times \Spa(\breve{E})^{\diamond} \lra \bB^{\varphi = \varpi^{h}} \times \Spa(\breve{E})^{\diamond} \lra \bA^{1,\diamond}_{\breve{E}} \lra 0 
\end{equation}
of diamonds which splits
pro-etale locally on $\bA^{1,\diamond}_{\breve{E}}$
as in 
\cite[Example 15.2.9 (4)]{ScWeBLp}. 
Therefore 
$\bB^{\varphi = \varpi^{h}} \times \Spa(\breve{E})^{\diamond}$ 
satisfies the claims \ref{en:Bcontr} and \ref{en:BactE}, since 
$\bA^{1,\diamond}_{\breve{E}}$ is 
an $\ell$-cohomologically smooth $\ell$-contractible diamond 
of pure dimension $1$ over $\Spa({\breve{E}})^{\diamond}$ 
and the action of $c \in E^{\times}$ on 
$\bA^{1,\diamond}_{\breve{E}}$ is induced from 
the isomorphism 
$\bA^{1}_{\breve{E}} \to \bA^{1}_{\breve{E}};\ x \mapsto cx$. 

The action of $a \in \bB^{\varphi = \varpi^{h}}(F^{\flat})$ 
on $f_{d,h,F^{\flat},!} \bZ_{\ell}$ depends only on  
the image $\overline{a} \in \bA^{1,\diamond}_{\breve{E}}(F^{\flat})$ of $a$ under \eqref{eq:BAext} 
since the claim \ref{en:Badd} is true for $\bB^{\varphi = \varpi^{h-1}}$. 
Hence it suffices to show that 
the action of $\overline{a}$ on 
$f_{\bA,!} \bZ_{\ell}$ is trivial, where 
$f_{\bA} \colon \bA^{1,\diamond}_{F} \to \Spa (F^{\flat})$ is the natural morphism. 
This follows from that the addition by 
$\overline{a}$ on 
$\bA^{1,\diamond}_{F}$ is induced from an automorphism on 
$\bA^{1}_{F}$ by \cite[Proposition 10.2.3]{ScWeBLp}. 
\end{proof}

Let $\delta_P \colon P(E) \to \ol{\bQ}_{\ell}^{\times}$ 
be the modulus character of $P(E)$. 
Let $A_b$ be the split center of $J_b$. 
Since $J_b$ is an inner form of $L^b$, 
we can view $A_b$ as an algebraic subgroup of $L^b$. 
We put $\delta_{P,A_b}=\delta_P|_{A_b(E)}$. 
Let $g \in J_b(E)$ act on $\wt{J}_b^U$ by 
the conjugation right action $u \mapsto g^{-1}ug$. 

\begin{lem}\label{lem:contr}
Let $f_J \colon \wt{J}_b^U \times \Spa({\breve{E}})^{\diamond} \to \Spa({\breve{E}})^{\diamond}$ be the natural morphism. 
\begin{enumerate}
\item\label{en:Jcontr} 
The functor $\wt{J}_b^U \times \Spa({\breve{E}})^{\diamond}$ 
is an $\ell$-cohomologically smooth $\ell$-contractible diamond of pure dimension $N_{U,b}$ over $\Spa({\breve{E}})^{\diamond}$. 
\item\label{en:JactE} 
Let $\kappa \colon J_b(E) \to \ol{\bQ}_{\ell}^{\times}$ 
be the character of the action of $J_b(E)$ on 
$f_{J,!} \ol{\bQ}_{\ell}$ induced by the conjugation right action of 
$J_b(E)$ on $\wt{J}_b^U$. 
Then we have $\kappa|_{A_b(E)}=\delta_{P,A_b}^{-1}$. 
\item\label{en:Jadd} 
Let $F$ be a perfectoid field over $\breve{E}$. 
Then the action of $\wt{J}_b^U(F^{\flat})$ on 
$f_{J,!} \ol{\bQ}_{\ell}$ induced by the addition on 
$\wt{J}_b^U$ is trivial. 
\end{enumerate}
\end{lem}
\begin{proof}
For $i \geq 0$, 
we define an algebraic subgroup $U_i$ of $P$ by 
\[
 U_i (R) = \Bigl\{ g \in P(R) \Bigm| g|_{\Fil_{\lambda}^j V_R} \equiv \id_{\Fil_{\lambda}^j V_R} \mod \Fil_{\lambda}^{j+i+1} V_R \ 
 \textrm{for all $j$ and $V \in \Rep (G)$} \Bigr\}  
\]
for any $E$-algebra $R$, where $V_R =V \otimes_E R$.  
Then $U_0=U$, 
and $U_i$ are normal in $P$ for all $i$. 
Similarly, we define 
a subsheaf $\wt{J}^U_{b,i}$ of $\wt{J}_b$ for $i \geq 0$ by 
\[
 \wt{J}^U_{b,i} (S)= \Bigl\{ g \in \wt{J}_b (S) \Bigm| g|_{\Fil_{\lambda}^j 
 \sE_b} \equiv \id_{\Fil_{\lambda}^j \sE_b} \mod \Fil_{\lambda}^{j+i+1} \sE_b \ 
 \textrm{for all $j$} \Bigr\}  
\]
for $S \in \Perf_{\ol{\bF}_q}$. 
Then $\wt{J}^U_{b,0} = \wt{J}^U_b$. 
Let $\varphi$ act on $G_{\breve{E}}$ and its subgroup 
$U_{i,\breve{E}}$ by 
$g \mapsto b_0 \sigma (g) b_0^{-1}$. 
Let $S$ be a perfectoid space over 
$\Spa({\breve{E}})^{\diamond}$. 
By the internal definition of a $G$-torsor 
on the Fargues--Fontaine curve, 
we see that 
$\wt{J}^U_{b,i}(S)$ is equal to the sections of 
\[
 Y_S \times_{\varphi} U_{i,\breve{E}} \lra X_S . 
\]
Hence, 
$(\wt{J}^U_{b,i}/\wt{J}^U_{b,i+1})(S)$ 
is equal to the sections of 
\[
 Y_S \times_{\varphi} 
 (U_{i,\breve{E}}/U_{i+1,\breve{E}}) \lra X_S. 
\] 
Let $L$ act on $U_i$ by the conjugation. 
Let $\Lie (G)$ be the adjoint representation of $G$. 
Then the action of $L$ on $\Lie (G)$ induces an action of 
$L$ on $\Lie (U_i / U_{i+1})$.
We have an isomorphism 
\[
 U_i / U_{i+1} \simeq \Lie ( U_i / U_{i+1} ) 
\]
as representations of $L$, 
since $U_i / U_{i+1}$ isomorphic to 
$\bG_a^{d_i}$ for some $d_i$ 
as linear algebraic groups. 
We have the equality 
\[
 \Lie (U_i) = \Fil_{\lambda}^i \Lie (G) 
\]
by the definition of the both sides. 
Hence we have an isomorphism 
\[
 \Lie ( U_i / U_{i+1} ) \simeq \gr_{\lambda}^i \Lie (G)
\]
as representations of $L$. 
As a result we have an isomorphism 
\begin{equation}\label{eq:UULie}
 U_i / U_{i+1} \simeq \gr_{\lambda}^i \Lie (G) 
\end{equation}
as representations of $L$. 
The element $b_0 \in L$ gives an $L$-bundle 
$\sE_{b_0,S} \colon \Rep (L) \to \Bun_{X_S}$. 
Then we have 
\[
 Y_S \times_{\varphi} (U_{i,\breve{E}}/U_{i+1,\breve{E}})
 \simeq \sE_{b_0,S} (\gr_{\lambda}^i \Lie (G) ) 
\] 
by \eqref{eq:UULie}. 
Hence,  
$(\wt{J}^U_{b,i}/\wt{J}^U_{b,i+1})(S)$ 
is equal to the sections of 
\[
 \sE_{b_0,S} (\gr_{\lambda}^i \Lie (G) ) \lra X_S . 
\] 
Then $\bD$ acts on $\gr_{\lambda}^i \Lie (G)$ 
via $\nu_b$ and the conjugation. 
This action gives a slope decomposition 
\[
 \gr_{\lambda}^i \Lie (G) = 
 \bigoplus_{1 \leq j \leq m_i} 
 V_{-\alpha_{i,j}} 
\] 
where 
$\alpha_{i,j}$ are positive rational numbers, 
since $L$ contains the centralizer $L^b$ of $\nu_b$. 
Then we have an isomorphism 
\begin{equation}\label{eq:decEgr}
  \sE_{b_0} (\gr_{\lambda}^i \Lie (G) ) \simeq 
 \bigoplus_{1 \leq j \leq m_i} 
 \cO(\alpha_{i,j}) .
\end{equation}
Hence  
$(\wt{J}^U_{b,i}/\wt{J}^U_{b,i+1}) \times \Spa({\breve{E}})^{\diamond}$ 
is an $\ell$-cohomologically smooth $\ell$-contractible diamond  
by \eqref{eq:decEgr} and Lemma \ref{lem:Bcontr}.

We show that 
$\wt{J}^U_{b,i} \times \Spa({\breve{E}})^{\diamond}$ is 
an $\ell$-cohomologically smooth $\ell$-contractible diamond 
by a decreasing induction on $i$. 
The claim is trivial for enough large $i$, 
since $\wt{J}^U_{b,i} \times \Spa({\breve{E}})^{\diamond}$ is one point for such $i$. 
We see that 
$U_{i,\breve{E}}$ is isomorphic to 
$U_{i+1,\breve{E}} \times (U_{i,\breve{E}}/U_{i+1,\breve{E}})$ 
as schemes over $U_{i,\breve{E}}/U_{i+1,\breve{E}}$
with actions of $\varphi$ 
by \cite[XXVI Proposition 2.1]{SGA3-3} and its proof. 
Hence, 
$\wt{J}^U_{b,i} \times \Spa({\breve{E}})^{\diamond}$ is isomorphic to 
$\wt{J}^U_{b,i+1} \times (\wt{J}^U_{b,i}/\wt{J}^U_{b,i+1}) \times \Spa({\breve{E}})^{\diamond}$ 
as diamonds over $(\wt{J}^U_{b,i}/\wt{J}^U_{b,i+1}) \times \Spa({\breve{E}})^{\diamond}$. 
Therefore, we see that 
$\wt{J}^U_{b,i}  \times \Spa({\breve{E}})^{\diamond} \to (\wt{J}^U_{b,i}/\wt{J}^U_{b,i+1}) \times \Spa({\breve{E}})^{\diamond}$ 
is an $\ell$-cohomologically smooth morphism 
with $\ell$-contractible geometric fiber, 
since $\wt{J}^U_{b,i+1} \times \Spa({\breve{E}})^{\diamond}$ 
is an $\ell$-cohomologically smooth $\ell$-contractible diamond 
by our induction hypothesis. 
Then we see that 
$\wt{J}^U_{b,i} \times \Spa({\breve{E}})^{\diamond}$ is 
an $\ell$-cohomologically smooth $\ell$-contractible diamond, 
since we know that 
$(\wt{J}^U_{b,i}/\wt{J}^U_{b,i+1}) \times \Spa({\breve{E}})^{\diamond}$ 
is an $\ell$-cohomologically smooth $\ell$-contractible diamond. 
The claim on the dimension follows from the above arguments. 
The claim \ref{en:JactE} 
follows from the arguments above, 
Lemma \ref{lem:Bcontr} \ref{en:BactE} and a calculation of $\delta_P$ (\cf  \cite[V.5.4]{RenReprp}). 
The claim \ref{en:Jadd} follows from 
Lemma \ref{lem:Bcontr} \ref{en:Badd} by induction on 
$i$ for $\wt{J}^U_{b,i}$ in the same way as the proof of Lemma \ref{lem:Bcontr} \ref{en:Badd}. 
\end{proof}

\begin{rem}
Some integral version of $\wt{J}_b$ 
is studied in \cite[Proposition 4.2.11]{CaScGenSV}. 
The character $\kappa$ in Lemma \ref{lem:contr} \ref{en:JactE} is explicitly determined in \cite[Corollary 4.6]{HaImDualcpx}. 
\end{rem}

Let 
$X_* (T)^{L+}$ 
be the set of $L$-dominant cocharacters 
in $X_* (T)$. 
We put 
\[
 I_{b_0,b'_0,\mu,L} = 
 \Bigl\{ 
 [\mu'] \in X_* (T)^{L+}/\Gamma \Bigm| 
 \textrm{$\mu'$ is $G$-conjugate to $\mu$ 
 and $[b_0] \in B(L,\mu',[b'_0])$}
 \Bigr\} . 
\]
We claim the set $I_{b_0,b_0',\mu,L}$ consists of a single element. To prove this we begin with a preliminary lemma.

\begin{lem} \label{lem:conweyl}
Given two cocharacters $\mu,\mu' \in X_{*}(T)$ which are $G$-conjugate, 
then there exists an element $w$ of the absolute Weyl group of $T$ in $G$ such that $w \cdot \mu = \mu'$.
\end{lem}

\begin{proof}
Let $L_{\mu}$ be the centralizer of the cocharacter $\bG_{m} \xrightarrow{\mu} T \rightarrow G$ and define similarly $L_{\mu'}$. Then, 
since $\mu' = g\mu g^{-1}$ 
for some $g \in G(\ol{E})$, 
it follows that $L_{\mu'} = gL_{\mu}g^{-1}$. 
Since $gTg^{-1} \subseteq L_{\mu'}$ 
is a maximal torus, there exists 
$l \in L_{\mu'}$ such that $gTg^{-1} = lTl^{-1}$. 
This means that $l^{-1}g$ normalizes $T$ and 
gives an element $w$ in the absolute Weyl group of $T$ in $G$. 
Then we have $w \cdot \mu = \mu'$. 
\end{proof}

\begin{lem}\label{lem:Ising}
$I_{b_0,b'_0,\mu,L}$ consists of a single element. 
\end{lem}
\begin{proof}
By the definition of 
Hodge--Newton reducibility, 
we have 
$[\mu] \in I_{b_0,b'_0,\mu,L}$. 
Let $[\mu'] \in I_{b_0,b'_0,\mu,L}$ be 
another element. 
Let $\Delta (G,T)$ be the set of simple roots of 
$G$ with respect to $T$, where the positivity of roots is given by $B$. 
Since $\mu$ is $G$-dominant, $\mu'$ is $G$-conjugate to $\mu$ 
and $\mu \neq \mu'$, 
we have that $\mu'$ is not $G$-dominant and 
\begin{equation}\label{eq:mudif}
 \mu -\mu' = \sum_{\alpha \in \Delta (G,T)} n_{\alpha} \alpha^{\vee} ,  
\end{equation}
where $n_{\alpha} \geq 0$ 
by Lemma \ref{lem:conweyl}, \cite[10.3 Lemma B]{HumIntL} and 
\cite[VI \S 1 Proposition 18]{BourLie46}. 
Since $\mu'$ is not $G$-dominant, but $L$-dominant, 
there is 
$\alpha_0 \in \Delta (G,T) \setminus \Delta (L,T)$ 
such that 
$\langle \mu' , \alpha_0 \rangle <0$. 
Then we have 
\begin{equation}\label{eq:mual0}
 \langle \mu - \mu' , \alpha_0 \rangle >0. 
\end{equation}
Substituting \eqref{eq:mudif} to \eqref{eq:mual0}, 
we have 
\[
 \sum_{\alpha \in \Delta (G,T)} n_{\alpha} 
 \langle \alpha^{\vee} , \alpha_0  \rangle >0 . 
\]
This implies $n_{\alpha_0} >0$, since 
we have 
$\langle \alpha^{\vee} , \alpha_0  \rangle \leq 0$ for 
$\alpha \neq \alpha_0$ by \cite[10.1 Lemma]{HumIntL}. 
Recall that 
\begin{equation} \label{eq:bordes}
\pi_{1}(L) = X_{*}(T) \big/ \textstyle \sum_{\alpha \in \Delta(L,T)} \bZ\alpha^{\vee}, 
\end{equation}
by the proof of \cite[Proposition 1.10]{BoroAGc} 
(\cf \cite[\S 1.13]{RaRiFiso}). 
Let $\ol{\mu}^{\natural}$ and $\ol{\mu'}^{\natural}$ 
be the images in $\pi_1 (L)_{\bQ}^{\Gamma}$ 
of 
$\ol{\mu}$ and $\ol{\mu'}$ in 
$X_*(T)_{\bQ}^{\Gamma}$. 

We show that 
$\ol{\mu}^{\natural} \neq \ol{\mu'}^{\natural}$. 
We write 
\[
 \ol{\mu} -\ol{\mu'} = 
 \sum_{\alpha \in \Delta (G,T)} m_{\alpha} \alpha^{\vee} , 
\]
where $m_{\alpha} \in \bQ$. 
Then the equation 
\[
 \ol{\mu} -\ol{\mu'} = [\Gamma: 
 \Gamma_{\mu} \cap \Gamma_{\mu'}]^{-1}
 \left( (\mu-\mu') + \sum_{1 \not= \tau \in \Gamma/(\Gamma_{\mu} \cap \Gamma_{\mu'})} \tau(\mu-\mu') \right) 
\]
implies $m_{\alpha_0} > 0$, 
since $n_{\alpha_0} > 0$ and 
$n_{\alpha} \geq 0$ for all 
$\alpha \in \Delta (G,T)$. 
Thus when passing to $\pi_1(L)^{\Gamma}$ the term $\alpha_{0}^{\vee}$ is not killed according to 
\eqref{eq:bordes} and so $\ol{\mu}^{\natural} \neq \ol{\mu'}^{\natural}$ as claimed. 
This implies 
\[
 \mu^{\natural} \neq \mu'^{\natural} \in \pi_1 (L)_{\Gamma}, 
\]
since 
$\ol{\mu}^{\natural}$ and $\ol{\mu'}^{\natural}$ are 
images of 
$\mu^{\natural}$ and $\mu'^{\natural}$ 
under the map 
\[
 \pi_1 (L)_{\Gamma} \to \pi_1 (L)_{\bQ}^{\Gamma} ; \ 
 [g] \mapsto \frac{1}{[\Gamma : \Gamma_{g}]} 
 \sum_{\tau \in \Gamma/\Gamma_g} \tau (g) , 
\] 
where $g \in \pi_1 (L)$ and $\Gamma_g$ is the stabilizer of 
$g$ in $\Gamma$. 
This contradicts that $[\mu'] \in I_{b_0,b'_0,\mu,L}$, 
because we have 
\[
 \mu'^{\natural}=\kappa_L ([b_0]) -\kappa_L ([b'_0]) =\mu^{\natural} 
 \in \pi_1 (L)_{\Gamma} 
\]
by $[b_0] \in B(L,\mu',[b'_0])$ and $[b_0] \in B(L,\mu,[b'_0])$. 
\end{proof}

\begin{defn}\label{def:ty}
Let $R$ be a DVR with uniformizer 
$\pi$, and quotient field $F$. 
Let $k_1 \geq \cdots \geq k_n$ 
be a sequence of integers. 
We say that the type of $g \in \GL_n (F)$ 
is $(k_1, \ldots , k_n)$ if we have 
\[
 g \in \GL_n (R) 
 \begin{pmatrix} 
  \pi^{k_1} & &  \\
  & \ddots &  \\
  & & \pi^{k_n} 
 \end{pmatrix} 
 \GL_n (R) . 
\] 
\end{defn}

\begin{lem}\label{lem:gLg}
Let $R$ be a DVR with uniformizer 
$\pi$, and quotient field $F$. 
We consider the subgroups 
\[
 L= 
 \begin{pmatrix} 
  \GL_{n_1} & &  \\
  & \ddots &  \\
  & & \GL_{n_m} 
 \end{pmatrix} 
 \subset 
  P= 
 \begin{pmatrix} 
  \GL_{n_1} & &  0  \\
   & \ddots &  \\
  * &  & \GL_{n_m} 
 \end{pmatrix} 
 \subset \GL_n 
\]
of $\GL_n $. 
Let $g \in P(F)$, 
and $g_L$ be the image of $g$ in the Levi quotient. 
We regard $g_L$ as an element of $L(F)$. 
We put $N_l =n_1 + \cdots +n_l$ for $0 \leq l \leq m$. 

Let $k_1 \geq \cdots \geq k_n$ 
be a sequence of integers. 
Assume that the type of 
\[
 (g_{ij})_{N_l +1 \leq i,j \leq n} \in \GL_{n-N_l} (F) 
\]
is 
$(k_{N_l +1}, \ldots ,k_n)$ 
for $0 \leq l \leq m-1$. 
Then we have 
$g_L^{-1} g \in P(R)$. 
\end{lem}
\begin{proof}
By multiplying a power of $\pi$ to $g$, 
we may assume that $k_n \geq 0$. 
By the assumption, 
we see that the type of 
\[
 (g_{ij})_{N_l +1 \leq i,j \leq N_{l+1}} \in \GL_{n_{l+1}} (F) 
\]
is 
$(k_{N_l +1}, \ldots ,k_{N_{l+1}})$ 
for $0 \leq l \leq m-1$ 
using Lemma \ref{lem:DVRquot}. 
Hence, we may assume that 
$g_L =\diag (\pi^{k_1} , \ldots , \pi^{k_n})$. 

Let $v$ be a normalized valuation of $F$. 
Then, 
it suffices to show that 
$v(g_{ij}) \geq k_i$ for all $1 \leq j < i \leq n$. 
Assume it does not hold, and 
take the biggest $i_0$ such that 
there is $j_0 <i_0$ satisfying $v(g_{i_0 j_0}) < k_{i_0}$. 
Then the type of 
\[
 (g_{ij})_{i_0 +1 \leq i,j \leq n} \in \GL_{n-i_0} (F) 
\]
is $(k_{i_0 +1}, \ldots ,k_n)$. 
Using this and Lemma \ref{lem:DVRquot}, 
we can show that 
the type of 
\[
 (g_{ij})_{1 \leq i,j \leq i_0} \in \GL_{i_0} (F) 
\]
is $(k_1, \ldots ,k_{i_0})$. 
This implies that 
$v(g_{ij}) \geq k_{i_0}$ for all 
$1 \leq i,j \leq i_0$. 
This contradicts the choice of $i_0$. 
\end{proof}

In the sequel, we simply write $(R^{\sharp},f)$ for 
\[
 (\sE_b, \sE_{b'}, R^{\sharp}, f, \id_{\sE_b}, \id_{\sE_{b'}} ) \in 
 \cM_{b,b'}^{\mu} (R,R^+). 
\] 
Every point of $\cM_{b,b'}^{\mu} (R,R^+)$ is represented by a datum of 
the above form, since 
we have an isomorphism of data 
\[
 (\sE, \sE', R^{\sharp}, f, \phi, \phi') \simeq 
 (\sE_b, \sE_{b'}, R^{\sharp}, \phi'^{-1} \circ f \circ \phi, \id_{\sE_b}, \id_{\sE_{b'}} ) 
\]
for 
\[
 (\sE, \sE', R^{\sharp}, f, \phi, \phi') \in \cM_{b,b'}^{\mu} (R,R^+). 
\] 
We write $D_{R^{\sharp}}$ for the degree $1$ Cartier divisor given by $R^{\sharp}$. 

We define a morphism 
\[
 \Phi \colon 
 \cM_{b_0,b'_0}^{\mu} \times \wt{J}_{b}^{U} 
 \lra \cP_{b,b'}^{\mu} 
\]
by sending 
\[
 \bigl( (R^{\sharp}, f_L ), g \bigr) \in 
 \Bigl( \cM_{b_0,b'_0}^{\mu} \times  \wt{J}_{b}^{U} 
 \Bigr) 
 (R,R^+) 
\]
to 
\[
 \bigl( R^{\sharp}, (f_L \times^L P) \circ g \bigr) 
 \in 
 \cP_{b,b'}^{\mu} (R,R^+) 
\]
for a perfectoid affinoid $\ol{\bF}_q$-algebra $(R,R^+)$. 

\begin{prop} \label{prop:Pdec}
The morphism 
\[
 \Phi \colon \cM_{b_0,b'_0}^{\mu} \times  \wt{J}_{b}^{U} 
 \lra \cP_{b,b'}^{\mu} 
\]
is an isomorphism. 
\end{prop}
\begin{proof}
Let $(R,R^+)$ be a perfectoid affinoid $\ol{\bF}_q$-algebra, and 
\[
 \bigl( (R^{\sharp}, f_L ), g \bigr) \in 
 \Bigl( \cM_{b_0,b'_0}^{\mu} \times  \wt{J}_{b}^{U} 
 \Bigr) 
 (R,R^+). 
\]
Then we have 
$\Phi \bigl( (R^{\sharp}, f_L ), g \bigr) \times^P L =(R^{\sharp}, f_L )$. 
Further, $(R^{\sharp}, f_L )$ and 
$\Phi \bigl( (R^{\sharp}, f_L ), g \bigr)$ recover $g$. 
Hence, we have the injectivity of $\Phi$. 

Let 
\[
 (R^{\sharp}, f ) \in \cP_{b,b'}^{\mu} (R,R^+). 
\] 
By the definition of $\cP_{b,b'}^{\mu}$, 
we have a reduction 
\[
 f_P \colon (\sE_{b_0} \times^L P) |_{X_R^{\sch} \setminus D_{R^{\sharp}}} 
 \stackrel{\sim}{\lra} 
 (\sE_{b'_0} \times^L P) |_{X_R^{\sch} \setminus D_{R^{\sharp}}} 
\]
of $f$ to $P$. 
We put 
$f_L = f_P \times^P L$. 

We show that 
\begin{equation}\label{eq:fLfJ}
 (f_L \times^L P)^{-1} \circ f_P \in \wt{J}_b^U (R,R^+) . 
\end{equation}
For this, it suffices to show \eqref{eq:fLfJ} 
after taking realizations for all $V \in \Rep (G)$. 
Hence, we may assume that $G=\GL_n$. 

We view $\GL_n$-bundles as vector bundles. 
We take 
the diagonal torus and 
the upper half 
Borel subgroup as $T$ and $B$. 
Then we have 
\[
 L= 
 \begin{pmatrix} 
  \GL_{n_1} & &  \\
  & \ddots &  \\
  & & \GL_{n_m} 
 \end{pmatrix} 
 \subset 
  P= 
 \begin{pmatrix} 
  \GL_{n_1} & &  0  \\
   & \ddots &  \\
  * &  & \GL_{n_m} 
 \end{pmatrix} 
 \subset \GL_n . 
\]
We write 
\[
 b_0 =(b_1 , \ldots , b_m), \ 
 b'_0 =(b'_1 , \ldots , b'_m) 
 \in 
 \GL_{n_1} (\breve{E}) \times \cdots \GL_{n_m} (\breve{E}).  
\]
Then we have a decomposition 
\[
 \sE_b =\bigoplus_{1 \leq i \leq m} \sE_{b_i} , \quad 
 \sE_{b'} =\bigoplus_{1 \leq i \leq m} \sE_{b'_i} 
\]
as vector bundles. 
We put 
\[
 \Fil^j \sE_b = \bigoplus_{j \leq i \leq m} \sE_{b_i}, \quad 
 \Fil^j \sE_{b'} = \bigoplus_{j \leq i \leq m} \sE_{b'_i} 
\]
for $1 \leq j \leq m+1$. 
Then 
$f \colon \sE_b|_{X_R^{\sch} \setminus D_{R^{\sharp}}} \to \sE_{b'}|_{X_R^{\sch} \setminus D_{R^{\sharp}}}$ 
respects these filtrations. 
We can write 
\[
 f=\bigoplus_{1 \leq i \leq j \leq m} f_{ij} 
 \colon \sE_b|_{X_R^{\sch} \setminus D_{R^{\sharp}}} \lra \sE_{b'}|_{X_R^{\sch} \setminus D_{R^{\sharp}}}, 
\]
where $f_{ij} \colon \sE_{b_i}|_{X_R^{\sch} \setminus D_{R^{\sharp}}} \to \sE_{b_j'}|_{X_R^{\sch} \setminus D_{R^{\sharp}}}$. 
Then the morphism 
\[
 f_{jj}^{-1} \circ f_{ij} \colon 
 \sE_{b_i}|_{X_R^{\sch} \setminus D_{R^{\sharp}}} \lra \sE_{b_j}|_{X_R^{\sch} \setminus D_{R^{\sharp}}} 
\]
extends to a morphism 
$\sE_{b_i} \to \sE_{b_j}$ 
by Lemma \ref{lem:gLg}. 
Hence we have \eqref{eq:fLfJ} 
(\cf the proof of \cite[Theorem 4.1]{HanHarr}). 

It remains to show that 
$(R^{\sharp},f_L) \in \cM_{b_0,b'_0}^{\mu} (R,R^+)$. 
It suffices to show that the type of the modification 
$f_L$ is equal to $\mu$ geometric fiberwisely. 
Let $\mu'$ be the type of $f_L$ at a geometric point of 
$\Spa (R,R^+)$. 
The type of $f_L \times^L G$ is equal to $\mu$ 
by \eqref{eq:fLfJ}. 
Hence, we have $\mu' =\mu$ by Lemma \ref{lem:Ising}. 
\end{proof}

For a diamond $\cD$ over $\Spa (\breve{E})^{\diamond}$, 
let $\cD_{\bC_p^{\flat}}$ denote 
$\cD \times_{\Spa (\breve{E})^{\diamond}} \Spa (\bC_p^{\flat})$. 
Let $\kappa \colon J_b(E) \to \ol{\bQ}_{\ell}^{\times}$ 
be the character in Lemma \ref{lem:contr}. 

\begin{lem}\label{lem:quotiso}
\begin{enumerate}
\item\label{en:limlissm}
We have 
\begin{equation}\label{eq:limnat!}
\varinjlim_{K' \subset J_{b'}(E)} R f_{K',\natural} ((f_{K'}^!\ol{\bQ}_{\ell})^{\vee}) \in D_{\lis}(\Spa (\bC_p^{\flat}),\ol{\bQ}_{\ell}) , 
\end{equation}
where 
$K'$ runs along compact open subgroups of $J_{b'}(E)$ and $f_{K'} \colon \cM_{b ,b', \bC_p^{\flat}}^{\mu} /K' \to \Spa (\bC_p^{\flat})$. 
Further we can regard this as an object of the derived category of smooth representations of $J_b(E) \times J_{b'}(E)$. 
\item\label{en:MPiso}
We have an isomorphism 
\[
 H_{\mathrm{c}}^i 
 \bigl( \cM_{b_0 ,b_0', \bC_p^{\flat}}^{\mu} , 
 \ol{\bQ}_{\ell} \bigr) \otimes \kappa 
 \stackrel{\sim}{\lra}  
 H_{\mathrm{c}}^{i + 2N_{U,b} } 
 \bigl( \cP_{b,b',\bC_p^{\flat}}^{\mu} , 
 \ol{\bQ}_{\ell} \bigr)  
\] 
as smooth representations of 
$J_b(E) \times L_{b'}(E)$. 
\end{enumerate}
\end{lem}
\begin{proof}
We can define 
$R\Gamma_{\mathrm{c}} (\cM_{b ,K,b', \bC_p^{\flat}}^{\mu})$ and 
$R\Gamma_{\mathrm{c}} (\cM_{b ,b',K', \bC_p^{\flat}}^{\mu})$ 
in the same way as \cite[\S 3]{ImaConv} replacing $\IC_{\mu}'$ by $j_{\mu,\natural}j_{\mu}^* \IC_{\mu}'$, where $j_{\mu} \colon \Hecke^{\mu} \to \Hecke^{\leq \mu}$. 
Then \eqref{eq:limnat!} coincides with 
\[
\varinjlim_{K' \subset J_{b'}(E)} R\Gamma_{\mathrm{c}} (\cM_{b ,b',K', \bC_p^{\flat}}^{\mu})
\]
up to shift and Tate twist. 
By the proof of \cite[Proposition IX.2.1]{FaScGeomLLC} and \cite[Corollary VI.6.6]{FaScGeomLLC}, 
we can show that 
\[
\varinjlim_{K \subset J_{b}(E)} R\Gamma_{\mathrm{c}} (\cM_{b ,K,b', \bC_p^{\flat}}^{\mu}), \ 
\varinjlim_{K' \subset J_{b'}(E)} R\Gamma_{\mathrm{c}} (\cM_{b ,b',K', \bC_p^{\flat}}^{\mu}) \in D_{\lis}(\Spa (\bC_p^{\flat}),\ol{\bQ}_{\ell}).  
\] 
In the same way as \cite[Proposition 3.16]{ImaConv}, we can show that 
\[
\varinjlim_{K \subset J_{b}(E)} R\Gamma_{\mathrm{c}} (\cM_{b ,K,b', \bC_p^{\flat}}^{\mu}) \cong  
\varinjlim_{K' \subset J_{b'}(E)} R\Gamma_{\mathrm{c}} (\cM_{b ,b',K', \bC_p^{\flat}}^{\mu}).  
\] 
Hence the claims in \ref{en:limlissm} follow. 

By \ref{en:limlissm}, 
we can regard $H_{\mathrm{c}}^i 
\bigl( \cM_{b_0 ,b_0', \bC_p^{\flat}}^{\mu} , 
\ol{\bQ}_{\ell} \bigr)$ as a smooth representation of $J_{b_0}(E) \times J_{b_0'}(E)$. 
Then claim \ref{en:MPiso} follows from Lemma \ref{lem:contr} and Proposition \ref{prop:Pdec}. 
\end{proof}

\begin{thm}\label{thm:HInd}
Assume that 
$([b],[b'],\mu)$ is Hodge--Newton reducible for $L$. 
Then we have an isomorphism 
\[
 H_{\mathrm{c}}^{i+2N_{U,b}}  
 \bigl( \cM_{b,b',\bC_p^{\flat}}^{\mu} , 
 \ol{\bQ}_{\ell} \bigr)
 \simeq 
 \Ind_{P_{b'}(E)}^{J_{b'}(E)} 
 H_{\mathrm{c}}^{i} 
 \bigl( \cM_{b_0 ,b_0',\bC_p^{\flat}}^{\mu} , 
 \ol{\bQ}_{\ell} \bigr) \otimes \kappa 
\]
as smooth $J_b(E) \times J_{b'}(E)$-representations. 
\end{thm}
\begin{proof}
This follows from 
Corollary \ref{cor:MPJ} and 
Lemma \ref{lem:quotiso}. 
\end{proof}

\begin{lem}\label{lem:actU}
Let $(R,R^+)$ be a perfectoid affinoid 
$\ol{\bF}_q$-algebra. 
Let 
\[
 (\sE ,\sE' ,R^{\sharp} ,f, \phi, \phi') \in \cM_{b,b'}^{\mu}(R,R^+). 
\] 
For any $g \in \ul{U_{b'} (E)}(R,R^+)$, 
there exists $h \in \wt{J}_b^U(R,R^{+})$ such that 
$g \circ f' =f' \circ h$, where 
we put 
\[
 f'=\phi'^{-1} \circ f \circ \phi \colon \sE_b|_{X_R^{\mathrm{sch}}\setminus D_{R^{\sharp}}} \to \sE_{b'}|_{X_R^{\mathrm{sch}}\setminus D_{R^{\sharp}}}. 
\] 
\end{lem}
\begin{proof}
Let 
$j \colon X_R^{\mathrm{sch}}\setminus D_{R^{\sharp}} \to 
 X_R^{\mathrm{sch}}$ 
be the open immersion. 
Let $V \in \Rep (G)$. 
We have an embedding 
\[
 \sE_b (V) \hookrightarrow 
 j_* j^* \sE_b (V) 
 \stackrel{\sim}{\lra} j_* j^* \sE_{b'} (V), 
\]
where the second isomorphism is induced by 
$f'$. 
We have an action of $g$ on 
$j_* j^* \sE_{b'} (V)$. 
It suffices to show that 
$g$ stabilizes 
$\Fil_{\lambda}^{i} \sE_b (V)$ and induces 
the identity on 
$\gr_{\lambda}^{i} \sE_b (V)$ for all $i$. 

We show this claim by a decreasing induction on $i$. 
For enough large $i$, 
we have $\Fil_{\lambda}^{i} \sE_b (V) =0$ and 
the claim is trivial for such $i$. 
Assume that the claim is true for $i +1$. 
We have the natural embedding 
\[
 \gr_{\lambda}^i \sE_b (V) \hookrightarrow 
 j_* j^* \gr_{\lambda}^i \sE_b (V) 
 \stackrel{\sim}{\lra} 
 j_* j^* \gr_{\lambda}^i \sE_{b'} (V) 
\]
where the second isomorphism is induced by 
$f'$. 
We have a commutative diagram 
\[
 \xymatrix{
 \gr_{\lambda}^i \sE_b (V) \ar@{^(->}[r] \ar@{->}[d]_g & 
 j_* j^* \gr_{\lambda}^i \sE_{b'} (V) \ar@{->}[d]^-{j_* j^* \gr_{\lambda}^i g} \\
 \bigl( g \Fil_{\lambda}^i \sE_b (V) \bigr) / \Fil_{\lambda}^{i+1} 
 \sE_b (V) \ar@{^(->}[r] & 
 j_* j^* \gr_{\lambda}^i \sE_{b'} (V) , 
 }
\]
where the bottom morphism is induced by 
the natural inclusion 
\[
 g \Fil_{\lambda}^i \sE_b (V) \subset 
 g \bigl( j_* j^* \Fil_{\lambda}^i \sE_{b'} (V) \bigr) = 
 j_* j^* \Fil_{\lambda}^i \sE_{b'} (V). 
\]
By this diagram, 
we see that 
$g \Fil_{\lambda}^i \sE_b (V) = \Fil_{\lambda}^i \sE_b (V)$, 
since $\gr_{\lambda}^i g$ is the identity 
on $\gr_{\lambda}^i \sE_{b'} (V)$. 
Hence, $g$ stabilizes 
$\Fil_{\lambda}^i \sE_b (V)$. 
Further, $g$ induces 
the identity on 
$\gr_{\lambda}^i \sE_b (V)$ 
again by the above diagram, 
since $\gr_{\lambda}^i g$ is the identity. 
\end{proof}

\begin{lem}\label{lem:Utriv}
The action of $U_{b'} (E)$ on 
$H_{\mathrm{c}}^i 
 \bigl( \cP_{b,b',\bC_p^{\flat}}^{\mu} , 
 \ol{\bQ}_{\ell} \bigr)$
is trivial. 
\end{lem}
\begin{proof}
Let $p_{\cM} \colon \cP_{b,b'}^{\mu} \cong \cM_{b_0,b'_0}^{\mu} \times \wt{J}_{b}^{U} \to 
\cM_{b_0,b'_0}^{\mu}$ be the projection, where the first isomorphism is given by Proposition \ref{prop:Pdec}. 
It suffices to show that the action of 
$U_{b'} (E)$ on $p_{\cM,!} \ol{\bQ}_{\ell}$ is trivial. 
It suffices to show this after the pullback to 
each geometric point of $\cM_{b_0,b'_0}^{\mu}$. 
It follows from Lemma \ref{lem:contr} \ref{en:Jadd} and Lemma \ref{lem:actU}. 
\end{proof}

\begin{prop}\label{prop:pivan}
Let $\pi$ be a smooth representation of $J_{b'}(E)$. 
Assume that 
$([b],[b'],\mu)$ is Hodge--Newton reducible for $L$ and 
that the Jacquet module of $\pi$ with respect to 
$P_{b'}$ vanishes. 
Then we have 
\[
 \Hom_{J_{b'}(E)} 
 \Bigl( \pi, H_{\mathrm{c}}^i 
 \bigl( \cM_{b,b',\bC_p^{\flat}}^{\mu} , 
 \ol{\bQ}_{\ell} \bigr) \Bigr) 
 =0.  
\]
\end{prop}
\begin{proof}
This follows from 
Theorem 
\ref{thm:HInd} and 
Lemma \ref{lem:Utriv}. 
\end{proof}

We define 
$ t_{b,b'} \colon \cT_{b,b',\bC_p^{\flat}}^{\mu} \to
 [\Spa (\ol{\bF}_q)/J_{b'}(E) ]$
as the composites 
\begin{align*}
 \cT_{b,b',\bC_p^{\flat}}^{\mu} \lra 
 \cT_{b,b'}^{\mu} \lra 
 \Hecke_{b,b'}^{\mu} \lra 
 [\Spa (\ol{\bF}_q)/J_{b'}(E) ]. 
\end{align*}
We put 
$\ola{t}_{b,b'} = x_{b'} \circ  t_{b,b'}$. 

\begin{thm}\label{thm:van}
Assume that $b$ is not basic 
and 
$([b],[b'],\mu)$ is Hodge--Newton reducible for $L$. 
Then we have 
\[
 H_{\mathrm{c}}^i 
 \bigl( \cT_{b,b',\bC_p^{\flat}}^{\mu} , 
 {\ola{t}}_{b,b'}^* \sF_{\varphi} 
 \bigr) =0. 
\]
\end{thm}
\begin{proof}
We have 
\[
 {\ola{t}}_{b,b'}^* \sF_{\varphi} 
 = t_{b,b'}^* x_{b'}^* \sF_{\varphi} 
 = t_{b,b'}^* 
 \left( \bigoplus_{\rho \in \wh{S}_{\varphi},\, \rho|_{Z(\wh{G})^{\Gamma}} =\kappa (b')} 
 \ul{\rho} \otimes \ul{\pi_{\varphi,b',\rho}} \right)
\]
by \eqref{eq:resxb}. 
We take $\rho \in \wh{S}_{\varphi}$ 
such that $\rho|_{Z(\wh{G})^{\Gamma}} =\kappa (b')$. 
Then it suffices to show that 
\[
 H_{\mathrm{c}}^i 
 \Bigl( \cT_{b,b',\bC_p^{\flat}}^{\mu} , 
 t_{b,b'}^* \ul{\pi_{\varphi,b',\rho}} 
 \Bigr) =0. 
\]
The pullback of 
$\ul{\pi_{\varphi,b',\rho}}$ to 
$\cM_{b,b'}^{\mu}$ is a constant sheaf, 
since the map 
$\cM_{b,b'}^{\mu} \to [\Spa (\ol{\bF}_q)/\ul{J_{b'}(E)}]$ 
factorizes via $\Spa(\ol{\bF}_q)$. 
Hence, there is a Hochschild--Serre spectral sequence
 \[
  H_i \Bigl(J_{b'}(E), 
  H^j_{\mathrm{c}}
  \bigl(\cM_{b,b',\bC_p^{\flat}}^{\mu} ,
  \ol{\bQ}_\ell\bigr) \otimes \pi_{\varphi,b',\rho} 
  \Bigr) 
  \Rightarrow 
  H^{j-i}_{\mathrm{c}}
  \Bigl( \cT_{b,b',\bC_p^{\flat}}^{\mu} , 
  t_{b,b'}^* \ul{\pi_{\varphi,b',\rho}} \Bigr) 
 \]
by \eqref{eq:HSsp} and Lemma \ref{lem:quotiso}. 
We show that 
\[
 H_i \bigl(J_{b'}(E), H^j_{\mathrm{c}}
 \bigl(\cM_{b,b',\bC_p^{\flat}}^{\mu} , 
 \ol{\bQ}_\ell\bigr) \otimes \pi_{\varphi,b',\rho}\bigr) =0 
\] 
for all $i$ and $j$. 
Take a projective resolution 
\[
 \cdots 
 \lra V_1 \lra V_0 \lra 
 H_{\mathrm{c}}^j 
 \bigl( \cP_{b,b',\bC_p^{\flat}}^{\mu} , 
 \ol{\bQ}_{\ell} \bigr) 
\]
as smooth $L_{b'}(E)$-representations. By Lemma \ref{lem:quotiso} and Theorem \ref{thm:HInd} we have
\[
 H^j_{\mathrm{c}}
 \bigl(\cM_{b,b',\bC_p^{\flat}}^{\mu} , 
 \ol{\bQ}_\ell\bigr) \simeq 
 \Ind_{P_{b'}(E)}^{J_{b'}(E)} H_{\mathrm{c}}^j 
 \bigl( \cP_{b,b',\bC_p^{\flat}}^{\mu} , 
 \ol{\bQ}_{\ell} \bigr) 
\]
as smooth $J_{b'}(E)$-representations. 
Moreover, the induction on the right-hand-side is parabolic 
by Lemma \ref{lem:Utriv}. 
Parabolic induction preserves projective objects, 
since 
it has a Jacquet functor as the right adjoint functor 
by Bernstein's second adjoint theorem 
(\cf \cite[Theorem 3]{BusReplH}) 
and 
the Jacquet functor is exact. 
Note also that 
parabolic induction is exact. 
Thus we obtain the projective resolution  
\[
 \cdots \lra 
 \Ind_{P_{b'}(E)}^{J_{b'}(E)} V_1 \lra 
 \Ind_{P_{b'}(E)}^{J_{b'}(E)} 
 V_0 \lra 
 H^j_{\mathrm{c}}
 \bigl(\cM_{b,b',\bC_p^{\flat}}^{\mu} , 
 \ol{\bQ}_\ell\bigr)
\]
as smooth $J_{b'}(E)$-representations.  
Finally the right adjoint of $-\otimes \pi_{\varphi,b',\rho}$ 
in the category of smooth $J_{b'}(E)$-representations is 
$-\otimes \pi_{\varphi,b',\rho}^*$, 
where $\pi_{\varphi,b',\rho}^*$ is the smooth dual of $\pi_{\varphi,b',\rho}$. 
Both functors are exact and so in particular $-\otimes \pi_{\varphi,b',\rho}$ preserves exact sequences and projective objects. Thus we obtain the projective resolution
\[
 \cdots \lra 
 \Ind_{P_{b'}(E)}^{J_{b'}(E)} V_1 \otimes \pi_{\varphi,b',\rho} \lra \Ind_{P_{b'}(E)}^{J_{b'}(E)} 
 V_0 \otimes \pi_{\varphi,b',\rho} \lra 
 H^j_{\mathrm{c}}
 \bigl(\cM_{b,b',\bC_p^{\flat}}^{\mu} , 
 \ol{\bQ}_\ell\bigr) \otimes \pi_{\varphi,b',\rho} . 
\]
Note that  
$P_{b'}$ is a proper parabolic subgroup of $J_{b'}$, 
since $b$ is not basic. 
For $i \geq 0$, we have 
\[
 \Bigl( \pi_{\varphi,b',\rho} \otimes 
 \Ind_{P_{b'}(E)}^{J_{b'}(E)} V_i \Bigr)_{J_{b'}(E)} 
 =0 , 
\]
since $\pi_{\varphi,b',\rho}$ is cuspidal. 
Hence we have the claim. 
\end{proof}

\section{Non-abelian Lubin--Tate theory}\label{sec:NALT}
Assume that $G=\GL_n$ 
and $\mu (z)= \diag (z,1, \ldots, 1)$. 
In this case, 
$S_{\varphi} =\bG_m$ and 
$\Hecke^{\leq \mu}=\Hecke^{\mu}$. 
We simply write 
$\pi_{\varphi,b}$ for $\pi_{\varphi,b,1}$ 
for any $[b] \in B(\GL_n)_{\basic}$. 
We put 
\[
 b_1 = 
 \begin{pmatrix} 
            0&0&\cdots &0& \varpi \\
           1&0&\cdots &0&0\\
           0&1&\cdots &0&0\\
           \vdots&\vdots&\ddots&\vdots&\vdots\\
           0&0&\cdots &1&0
 \end{pmatrix} 
 \in \GL_n (E) . 
\]
Then we have a bijection 
\[
 \bZ
 \stackrel{\sim}{\lra} B(\GL_n)_{\basic} ;\ 
 N \mapsto b_1^N . 
\]
The following proposition is a consequence of non-abelian Lubin--Tate theory. 
\begin{prop}\label{prop:NALT}
We put $b=b_1^N$ for an integer $N$. 
Assume that $N \equiv 0, 1 \mod n$. 
Then we have 
\begin{equation}\label{eq:yb} 
 y_b^* \ora{h}_{\natural} (\ola{h}^{*} \sF_{\varphi} \otimes \IC_{\mu}') 
 = y_b^* (\sF_{\varphi} \boxtimes \varphi ). 
\end{equation}
\end{prop}
\begin{proof}
We show the claim 
in the case where $N \equiv 1 \mod n$ 
using arguments in \cite[Chapter 23]{MFOGPF}. 
See arguments in \cite[8.1]{FarGover} 
for the case where $N \equiv 0 \mod n$. 
Suppose that $N = mn+1$ for some $m \in \bZ$. 
The following lemma provides an explicit description of the stack $\Hecke^{\leq \mu}_{b}$. 

\begin{lem} \label{lem:TLT}
Let $\Spa (F,F^{+})$ be a geometric point in 
$\Perf_{\ol{\bF}_q}$. 
Let $\sE$ be a vector bundle of rank $n$ 
on $ X_{F}^{\mathrm{sch}}$ having a degree one modification fiberwise by $\sE_b$ 
\[
 0 \to \sE_b \to \sE \to \sF \to 0, 
\]
where $\sF$ is a torsion coherent sheaf of length $1$. 
Then $\sE$ is isomorphic to $\cO(-m)^n$. 
\end{lem}

\begin{proof}
This follows from 
\cite[Theorem 2.94]{FaFoVbp} by 
dualizing the modification and twisting by $\cO(-m)$. 
\end{proof}

We put $b'=b_1^{nm}$. 
Then, 
we have isomorphisms 
\begin{align*}
 \Hecke_{b,b'}^{\leq \mu} \simeq 
 \Hecke_b^{\leq \mu} 
\end{align*}
by Lemma \ref{lem:TLT}. 

\begin{lem}\label{lem:Mbb'LT}
Let $\cM_{\mathrm{LT}}^{\infty}$ 
be the Lubin--Tate space over $\breve{E}$ 
at infinite level. 
Then we have an isomorphism 
$\cM_{b,b'}^{\leq \mu} \simeq 
 \cM_{\mathrm{LT}}^{\infty,\diamond}$, 
that is compatible with actions of 
$\GL_n(E) \times J_b (E)$ and Weil descent data. 
\end{lem}
\begin{proof}
For a perfectoid affinoid $\ol{\bF}_q$-algebra $(R,R^+)$, 
the set 
$\cM_{b,b'}^{\leq \mu} (R,R^+)$
consists of 6-tuples 
$(\sE, \sE', R^{\sharp}, f,\phi,\phi')$, 
where 
\begin{itemize}
 \item $(\sE,\sE', D_{R^{\sharp}},f) \in \Hecke^{\leq \mu}_{b(0)}$ 
 \item $\phi \colon \sE_b \xrightarrow{\sim} \sE$ 
 and $\phi' \colon \sE_{b'} \xrightarrow{\sim} \sE'$ are isomorphisms. 
\end{itemize}
Hence, the claim follows from 
\cite[Proposition 6.3.9]{ScWeMpd} by 
dualizing the modification and twisting by $\cO(-m)$. 
\end{proof}

Let 
\begin{equation}\label{eq:pb}
 p_b \colon \Spa (\bC_p^{\flat}) \lra 
 \Spa ( \breve{E} )^{\diamond} \lra 
 [\Div_{X}^1 /\wt{J}_b ]  
\end{equation}
be the natural projection. 
The equality \eqref{eq:yb} is equivalent to 
the equality 
\begin{equation}\label{eq:ptyb}
 p_b^* y_b^* \ora{h}_{\natural} (\ola{h}^{*} \sF_{\varphi} \otimes \IC_{\mu}') 
 = p_b^* y_b^* (\sF_{\varphi} \boxtimes \varphi ) 
\end{equation}
with action of $J_b(E) \times W_E$. 
Then the right hand side of 
\eqref{eq:ptyb} is 
$\pi_{\varphi,b} \otimes \varphi$ 
as a representation of $J_b(E) \times W_E$. 
Hence 
it suffices to show that 
the cohomology of the left hand side of \eqref{eq:ptyb} 
vanishes outside degree zero, and 
is equal to 
$\pi_{\varphi,b} \otimes \varphi$ in degree zero 
as representations of $J_b(E) \times W_E$. 

The $i$-th cohomology of 
the left hand side of \eqref{eq:ptyb} is 
equal to 
\[
 H_{\mathrm{c}}^{i+n-1} 
 \bigl( \cT_{b,b',\bC_p^{\flat}}^{\leq \mu} , 
 \ola{t}_{b,b'}^* \sF_{\varphi} \bigr) 
 \Bigl( \frac{n-1}{2} \Bigr) . 
\]
We have 
\[
 \ola{t}_{b,b'}^* \sF_{\varphi} = 
 t_{b,b'}^* \ul{\pi_{\varphi,1}}  
\]
by \eqref{eq:resxb}, since 
$\pi_{\varphi,b'}=\pi_{\varphi,1}$ in our case. 
We have a Hochschild--Serre spectral sequence
 \[
  H_i \Bigl( \GL_n (E), 
  H^j_{\mathrm{c}}
  \bigl(\cM_{b,b',\bC_p^{\flat}}^{\leq \mu} ,
  \ol{\bQ}_\ell\bigr) \otimes \pi_{\varphi,1} 
  \Bigr) 
  \Rightarrow 
  H^{j-i}_{\mathrm{c}}
  \Bigl( \cT_{b,b',\bC_p^{\flat}}^{\leq \mu} , 
  t_{b,b'}^* \ul{\pi_{\varphi,1}} \Bigr)  
\]
by \eqref{eq:HSsp} and Lemma \ref{lem:quotiso}. 
We put 
\[
 \GL_n(E)^{0} =\{ g \in \GL_n(E) \mid \det (g) \in \cO_E^{\times} \} . 
\]
Then we have 
\[
 H^j_{\mathrm{c}}
 \bigl(\cM_{\mathrm{LT},\bC_p^{\flat}}^{\infty,\diamond}, 
 \ol{\bQ}_\ell\bigr) = \mathrm{c}\mathchar`-\Ind_{\GL_n(E)^{0}}^{\GL_n(E)}  
 H^j_{\mathrm{c}}
 \bigl(\cM_{\mathrm{LT},\bC_p^{\flat}}^{\infty,(0),\diamond}, 
 \ol{\bQ}_\ell\bigr)
\]
for a connected component 
$\cM_{\mathrm{LT}}^{\infty,(0)}$ of 
$\cM_{\mathrm{LT}}^{\infty}$ (\cf \cite[4.4.2]{FarCohpdiv}). 
By Lemma \ref{lem:Mbb'LT}, we have
\begin{align*}
H^j_{\mathrm{c}}
  \bigl(\cM_{b,b',\bC_p^{\flat}}^{\leq \mu} ,
  \ol{\bQ}_\ell\bigr) \otimes \pi_{\varphi,1} &=  \Bigl( \mathrm{c}\mathchar`-\Ind_{\GL_n(E)^{0}}^{\GL_n(E)}  
 H^j_{\mathrm{c}}
 \bigl(\cM_{\mathrm{LT},\bC_p^{\flat}}^{\infty,(0),\diamond}, 
 \ol{\bQ}_\ell\bigr) \Bigr) \otimes \pi_{\varphi,1} \\
 &= \mathrm{c}\mathchar`-\Ind_{\GL_n(E)^{0}}^{\GL_n(E)} \Bigl(H^j_{\mathrm{c}}
 \bigl(\cM_{\mathrm{LT},\bC_p^{\flat}}^{\infty,(0),\diamond}, 
 \ol{\bQ}_\ell\bigr) \otimes \pi_{\varphi,1}|_{\GL_n(E)^{0}} \Bigr).
\end{align*}
Therefore one has 
\[
  H_i \Bigl( \GL_n (E), 
  H^j_{\mathrm{c}}
  \bigl(\cM_{b,b',\bC_p^{\flat}}^{\leq \mu} ,
  \ol{\bQ}_\ell\bigr) \otimes \pi_{\varphi,1} 
  \Bigr) =   H_i \Bigl( \GL_n (E)^{0}, 
 H^j_{\mathrm{c}}
 \bigl(\cM_{\mathrm{LT},\bC_p^{\flat}}^{\infty,(0),\diamond}, 
 \ol{\bQ}_\ell\bigr) \otimes \pi_{\varphi,1}|_{ \GL_n (E)^{0}} 
  \Bigr)
\]
by Shapiro's Lemma. 
Now $ \pi_{\varphi,1}|_{ \GL_n (E)^{0}}$ is a compact representation and thus it is a projective object in the category of smooth $\GL_n (E)^{0}$-representations. Hence no higher homology groups appear and so 
\[
 \Bigl(  
 H^j_{\mathrm{c}}
 \bigl(\cM_{\mathrm{LT},\bC_p^{\flat}}^{\infty,\diamond}, 
 \ol{\bQ}_\ell\bigr) \otimes \pi_{\varphi,1} 
 \Bigr)_{\GL_n (E)} 
 =
 H^{j}_{\mathrm{c}}
 \Bigl( \cT_{b,b',\bC_p^{\flat}}^{\leq \mu} , 
 t_{b,b'}^* \ul{\pi_{\varphi,1}} \Bigr) . 
\]
Hence, the claim follows from the 
non-abelian Lubin--Tate theory. 
\end{proof}

\section{Hecke eigensheaf property}\label{sec:Hep}
Assume that $G=\GL_2$ 
and $\mu (z)=\diag (z,1)$ in this section. 

\begin{lem}\label{lem:E1nE2s}
Let $\Spa (F,F^+)$ be a geometric point in $\Perf_{\bF_q}$. 
Let 
\[
 0 \lra \sE \lra \sE' \lra
 \sF \lra 0 
\]
be an exact sequence of coherent sheaf over 
$X_F^{\mathrm{sch}}$, 
where $\sE$ and $\sE'$ are vector bundles of rank $2$ 
and $\sF$ is a torsion coherent sheaf of 
length $1$. 
Assume that 
$\sE$ is not semi-stable 
and 
$\sE'$ is semi-stable. 
Then 
$\sE \simeq \cO (m) \oplus \cO (m-1)$ 
and  
$\sE' \simeq \cO (m) \oplus \cO (m)$ 
for some integer $m$. 
\end{lem}
\begin{proof}
The vector bundle $\sE'$ is isomorphic to 
$\cO (m+\frac{1}{2})$ or 
$\cO (m) \oplus \cO (m)$ for some integer $m$, 
since it is semi-stable. 

If $\sE'$ is isomorphic to 
$\cO (m+\frac{1}{2})$, 
then 
$\sE$ 
is isomorphic to 
$\cO (m) \oplus \cO (m)$ 
by \cite[Theorem 2.9]{FaFoVbp}. 
This contradict to the condition that 
$\sE$ is not semi-stable. 

Assume 
$\sE'$ 
is isomorphic to 
$\cO (m) \oplus \cO (m)$. 
Then 
$\sE$ 
is isomorphic to 
$\cO (m_1) \oplus \cO (m_2)$ 
with $m_1, m_2 \leq m$ or 
$\cO (n+\frac{1}{2})$ 
with $n \leq m-1$ 
by \cite[6.3.1]{FaFoVbp}. 
By considering 
$\deg (\sE) +1=\deg (\sE')$, 
the possible cases are 
$\cO (m) \oplus \cO (m-1)$ 
or 
$\cO (m-\frac{1}{2})$. 
However, the latter case does not happen, 
since $\sE$ is not semi-stable. 
\end{proof}

\begin{prop}\label{prop:supp}
Then we have 
\[
 \supp \ora{h}_{\natural}(\ola{h}^{*}\sF_{\varphi} \otimes \IC_{\mu}') \subset 
 \Bun_{G}^{\mathrm{ss}} \times 
 \Div_X^1 . 
\]
\end{prop}
\begin{proof}
Take a non-basic element $[b] \in B(G)$. 
Then it suffices to show that 
$p_b^* y_b^* \ora{h}_{\natural} \ola{h}^* \sF_{\varphi} =0$, 
where $p_b$ is defined at \eqref{eq:pb}. 
We consider the following cartesian diagram: 
\begin{equation*}\label{eq:TbssTb}
 \xymatrix{
 \cT_{b,\bC_p^{\flat}}^{\leq \mu,\mathrm{ss}} 
 \ar@{->}[r] \ar@{->}[dd]_-{\ola{h}_{b,\mathrm{ss}}} & 
 \cT_{b,\bC_p^{\flat}}^{\leq \mu} \ar@{->}[r] \ar@{->}[d] & 
 \Spa ( \bC_p^{\flat} )  
 \ar@{->}^-{y_b \circ p_b}[d] \\
 & \Hecke^{\leq \mu} 
 \ar@{->}^-{\overrightarrow{h}}[r] 
 \ar@{->}^-{\overleftarrow{h}}[d] & 
 \Bun_{G} \times \Div_X^1 \\ 
 \Bun_{G}^{\mathrm{ss}} \ar@{->}[r]^-{j_{\mathrm{ss}}} & 
 \Bun_{G}. &  
 }
\end{equation*}
Let 
$\ola{h}_{b} \colon 
 \cT_{b,\bC_p^{\flat}}^{\leq \mu} 
 \to \Bun_{G}$ 
be the morphism which appears in the above diagram. 
Then it suffices to see that 
\[
 H_{\mathrm{c}}^i \bigl( \cT_{b,\bC_p^{\flat}}^{\leq \mu}, 
 {\ola{h}}_{b}^* \sF_{\varphi} 
 \bigr) =0. 
\]
On the other hand, we have 
\[
 H_{\mathrm{c}}^i \bigl( \cT_{b,\bC_p^{\flat}}^{\leq \mu}, 
 {\ola{h}}_{b}^* \sF_{\varphi} 
 \bigr) = 
 H_{\mathrm{c}}^i \bigl( 
 \cT_{b,\bC_p^{\flat}}^{\leq \mu,\mathrm{ss}}, 
 {\ola{h}}_{b,\mathrm{ss}}^* j_{\mathrm{ss}}^* 
 \sF_{\varphi} 
 \bigr) 
\]
by 
$\sF_{\varphi} = j_{\mathrm{ss},\natural} 
 j_{\mathrm{ss}}^* \sF_{\varphi}$. 
We have a decomposition 
\[
 \cT_{b,\bC_p^{\flat}}^{\leq \mu, \mathrm{ss}} = 
 \coprod_{N \in 2 \bZ} 
 \cT_{b,b_1^N,\bC_p^{\flat}}^{\leq \mu}
\]
by Lemma \ref{lem:E1nE2s}. 
Hence, we have 
\[
  H_{\mathrm{c}}^i \bigl( 
 \cT_{b,\bC_p^{\flat}}^{\leq \mu, \mathrm{ss}}, 
 {\ola{h}}_{b,\mathrm{ss}}^* j_{\mathrm{ss}}^* 
 \sF_{\varphi} 
 \bigr) =0 
\]
by Theorem \ref{thm:van}. 
\end{proof}

\begin{thm}
Then we have 
\[
 \ora{h}_{\natural}(\ola{h}^{*}\sF_{\varphi} \otimes \IC_{\mu}') 
 \cong \sF_{\varphi} \boxtimes \varphi . 
\]
\end{thm}
\begin{proof}
By Proposition \ref{prop:supp}, 
it suffices to show the equality on 
$\Bun_{G}^{\mathrm{ss}} 
 \times \Div_X^1$. 
The equality on 
the semi-stable locus follows from 
Proposition \ref{prop:NALT}, 
since we have $N \equiv 0,1 \mod 2$ 
for any integer $N$. 
\end{proof}

\noindent
Ildar Gaisin\\ 
Department of Mathematics, HSE University, Usacheva str. 6, 119048 Moscow, Russia\\ 
igaisin@hse.ru

\vspace*{10pt}

\noindent
Naoki Imai\\ 
Graduate School of Mathematical Sciences, 
The University of Tokyo, 3-8-1 Komaba, Meguro-ku, 
Tokyo, 153-8914, Japan\\ 
naoki@ms.u-tokyo.ac.jp 

\end{document}